\documentclass{amsart}

\DeclareRobustCommand{\SkipTocEntry}[5]{}
\usepackage{pxfonts}
\usepackage[T1]{fontenc}
\usepackage[utf8]{inputenc}

\usepackage{amssymb,amsthm,mathtools,mathrsfs,stmaryrd}
\usepackage[bookmarks=true,
colorlinks=true, citecolor=blue,
linkcolor=blue, linktocpage=true]{hyperref}
\usepackage{bbm}
\usepackage{enumitem}

\usepackage{aliascnt}
\newaliascnt{thm}{equation}
\usepackage[capitalize]{cleveref}
\newtheorem{thm}[equation]{Theorem}
\newtheorem{thm*}{Theorem}
\newtheorem{pro}[thm]{Proposition}
\newtheorem{hyp}[thm]{Hypothesis}
\newtheorem{lem}[thm]{Lemma}
\newtheorem{cor}[thm]{Corollary}
\theoremstyle{definition}

\newtheorem{con}[thm]{Convention}
\newtheorem{exa}[thm]{Example}
\newtheorem{rmk}[thm]{Remark}
\numberwithin{equation}{section}

\usepackage[numbers,sort&compress]{natbib}
\setcitestyle{notesep={: },citesep={;}}

\usepackage{xypic}
\usepackage{tikz}
\usetikzlibrary{backgrounds,arrows,shapes,decorations.pathmorphing,arrows.meta}
\usepackage{pgfplots}

\DeclareMathOperator{\car}{char}

\DeclareMathOperator{\cone}{cone}

\DeclareMathOperator{\dm}{DM_{\mathrm{gm}}}
\DeclareMathOperator{\dmeff}{DM_{\mathrm{gm}}^{\mathrm{eff}}}
\DeclareMathOperator{\dmgen}{D(T)M_{\mathrm{gm}}^{(\et)}}
\DeclareMathOperator{\DM}{DM}
\DeclareMathOperator{\DMeff}{DM^{\mathrm{eff}}}
\DeclareMathOperator{\dmet}{DM_{\mathrm{gm}}^{\et}}
\DeclareMathOperator{\dmeteff}{DM_{\mathrm{gm}}^{\et,eff}}
\DeclareMathOperator{\DMet}{DM^{\mathrm{\et}}}
\DeclareMathOperator{\DMeteff}{DM^{\mathrm{\et,eff}}}
\DeclareMathOperator{\dtm}{DTM_{\mathrm{gm}}}
\DeclareMathOperator{\DTM}{DTM}
\DeclareMathOperator{\dtmDTM}{DTM_{(\mathrm{gm})}}
\DeclareMathOperator{\dtmet}{DTM^{\et}_{\mathrm{gm}}}

\DeclareMathOperator{\End}{End}
\DeclareMathOperator{\fil}{\mathcal{F}}
\DeclareMathOperator{\filgal}{mod_{\scriptscriptstyle\mathrm{fil}}}
\DeclareMathOperator{\filgalu}{mod_{\scriptscriptstyle\mathrm{fil-un}}}

\DeclareMathOperator{\gal}{\mdfin}

\DeclareMathOperator{\Hm}{H}

\DeclareMathOperator{\Md}{Mod}
\DeclareMathOperator{\mdfin}{mod}

\DeclareMathOperator{\spc}{Spc}
\DeclareMathOperator{\spec}{Spec}
\DeclareMathOperator{\spch}{Spc^{\mathrm{h}}}
\DeclareMathOperator{\spech}{Spec^{\mathrm{h}}}
\DeclareMathOperator{\supp}{supp}

\providecommand{\A}{\ensuremath{\mathcal A}}
\providecommand{\DA}{\ensuremath{\mathbb A}}

\providecommand{\aet}{\ensuremath{a_{\et}}}
\providecommand{\bc}{\ensuremath{\gamma^{*}}}
\providecommand{\bco}{\ensuremath{\overline{\gamma}^{*}}}
\providecommand{\Cat}{\ensuremath{{\mathrm{Cat}}}}
\providecommand{\cf}{cf.\ }
\providecommand{\eg}{e.g.\ }
\providecommand{\et}{\ensuremath{\mathrm{\acute{e}t}}}
\providecommand{\id}{\ensuremath{\mathrm{id}}}
\providecommand{\ie}{i.e.\ }
\providecommand{\F}{\ensuremath{\mathbb{F}}}
\providecommand{\ff}{\ensuremath{\gamma_{*}}}
\providecommand{\ffo}{\ensuremath{\overline{\gamma}_{*}}}
\providecommand{\DF}{\ensuremath{\mathbb F}}
\providecommand{\gr}{\ensuremath{\mathrm{gr}}}

\providecommand{\Hmm}{\ensuremath{\mathrm{H}_{\scriptscriptstyle\mathrm{M}}}}
\providecommand{\Hmet}{\ensuremath{\mathrm{H}_{\et}}}
\providecommand{\I}{\ensuremath{\mathcal I}}

\providecommand{\Lrs}{\ensuremath{{\mathrm{LRS}}}}
\DeclareSymbolFont{bbold}{U}{bbold}{m}{n}
\DeclareSymbolFontAlphabet{\mathbbold}{bbold}
\providecommand{\one}{\ensuremath{\mathbbold{1}}}

\providecommand{\pos}{\ensuremath{\mathrm{pos}}}
\providecommand{\PP}{\ensuremath{\mathbbold P}}
\providecommand{\prmod}[1]{\ensuremath{\mathfrak{m}_{#1}}}
\providecommand{\prrat}{\ensuremath{\mathfrak{m}_{0}}}
\providecommand{\pret}[1]{\ensuremath{\mathfrak{e}_{#1}}}
\providecommand{\Q}{\ensuremath{\mathbb{Q}}}
\providecommand{\R}{\ensuremath{\mathcal R}}
\providecommand{\reet}{\ensuremath{\mathrm{Re}}}
\providecommand{\rigttCat}{\ensuremath{\mathrm{tt}\Cat^{\mathrm{rig}}}}
\providecommand{\rigttCatop}{\ensuremath{\mathrm{tt}\Cat^{\mathrm{rig,op}}}}
\providecommand{\Rng}{\ensuremath{{\mathrm{Rng}}}}
\providecommand{\Rs}{\ensuremath{{\mathrm{RS}}}}

\providecommand{\sm}{\ensuremath{\mathrm{Sm}}}
\providecommand{\T}{\ensuremath{\mathcal T}}
\providecommand{\DT}{\ensuremath{\mathbb T}}
\providecommand{\Tpl}{\ensuremath{{\mathrm{Top}}}}

\providecommand{\Z}{\ensuremath{\mathbb{Z}}}
\providecommand{\cal}{\mathcal}

\DeclarePairedDelimiter{\delimpar}{(}{)}
\newcommand{\prns}[2][0]{%
  \ifcase#1\relax
    \delimpar{#2}\or
    \delimpar[\big]{#2}\or
    \delimpar[\Big]{#2}\or
    \delimpar[\bigg]{#2}\or
    \delimpar[\Bigg]{#2}
  \else
    \delimpar*{#2}
  \fi}

\makeatletter
\def\scripts#1#2#3{\def\scripts@{\prns[#1]{#3}}\def\scripts@@{#2}\def\scripts@@@{#2}\@ifnextchar^\@sup\@nsup}
\def\@sup^#1{\def\scripts@@@{\scripts@@^{#1}}\@ifnextchar_\@supsub{\scripts@@@\scripts@}}
\def\@supsub_#1{\scripts@@@_{#1}\scripts@}
\def\@nsup{\@ifnextchar_{\@sub}{\scripts@@@\scripts@}}
\def\@sub_#1{\def\scripts@@@{\scripts@@_{#1}}\@ifnextchar^\@subsup{\scripts@@@\scripts@}}
\def\@subsup^#1{\scripts@@@^{#1}\scripts@}
\makeatother

\newcommand{\D}[2][0]{\scripts{#1}{\operatorname{\mathrm{D}}}{#2}}

\begin{document}

\title{tt-geometry of Tate motives over algebraically closed fields}
\author{Martin Gallauer}
\email{gallauer@maths.ox.ac.uk}
\address{Mathematical Institute, University of Oxford, Oxford, OX2 6GG, UK}
\urladdr{http://people.maths.ox.ac.uk/gallauer}

\begin{abstract}
  We study Tate motives with integral coefficients through the lens of tensor triangular geometry. For some base fields, including $\overline{\Q}$ and $\overline{\F_{p}}$, we arrive at a complete description of the tensor triangular spectrum and a classification of the thick tensor ideals.
\end{abstract}
\maketitle
\tableofcontents
\section{Introduction}
\label{sec:introduction}
Although the theory of (mixed) motives has in recent years made significant progress, especially in regards to motivic sheaves (\ie motives over general schemes), our understanding of motives over a field is still limited. It is therefore natural to restrict one's attention to certain subclasses of motives in order to gain some intuition. One subclass which has proven particularly fruitful for that purpose is the class of \emph{(mixed) Tate motives}. These are the motives which can be constructed from the simple building blocks $\Z(n)$ (the ``Tate twists'', $n\in\Z$) by extensions, (de)suspensions, and direct summands. Another view is that they encode the motivic cohomology of the base field. As should be clear from these two descriptions, Tate motives are at the same time relatively simple (compared to the class of all motives), and yet contain complex and interesting information. A striking illustration of the latter fact is their role in the theory of periods, particularly their relation to multiple zeta values, as explained in~\cite{deligne-goncharov:tate-fundamental-group,terasoma:tate-motives-multiple-zeta,brown:tate-motives-over-Z}.

One long strand of research into Tate motives has been concerned with uncovering the \emph{structural properties} of the category of Tate motives, with some success \cite[\eg][]{bloch:tate-motives,bloch-kriz:tate-motives,levine92-tatemotives,kriz-may:operads-motives,positselski:artin-tate-motives,wildeshaus:tate,iwanari:bar-tannakization}. The motivation for the present paper is to advance in this direction, by focusing on the particular properties visible to \emph{tensor triangular geometry}. Indeed, Tate motives form a triangulated category with a compatible tensor structure (a \emph{tt-category} for short), and one may try to classify its tt-ideals, \ie its thick subcategories closed under tensor products with arbitrary Tate motives. The classification of tt-ideals gives insights into the composition and complexity of the tt-category; if two objects generate different tt-ideals, this means that they cannot be constructed out of each other using extensions, (de)suspensions, tensor products, and direct summands: they are quite far from being ``the same''.

The main device in tt-geometry is Balmer's \emph{tt-spectrum}~\cite{balmer:spectrum}, a spectral space associated to the tt-category which encodes the (radical) tt-ideals through its topology. For a number of classical mathematical domains, the tt-spectrum has been studied; we refer to~\cite{balmer:icm} for an overview of the basic theory, its early successes and applications. In relation to the present article, there has been earlier work on the tt-spectrum of certain motives (including Tate motives)~\cite{peter:spectrum-damt,kelly:tt-motives-finite-fields}. However, the arguments were restricted to rational coefficients (and certain base fields) so that the tt-spectrum turned out to be a singleton space. In other words, in these categories every object can be constructed from every other (non-zero) object.

In contrast, here we work with integral coefficients and find non-trivial tt-spectra. In particular, we determine completely the tt-spectrum of Tate motives over the algebraic numbers (for more general base fields see \cref{dtm-primes} and \cref{closed-subsets}):

\begin{thm*}\label{intro-dtm-Qbar}
The tt-spectrum of $\dtm(\overline{\Q},\Z)$ consists of the following points, with specialization relations, depicted by the lines, going upward.
\vspace{0.2cm}
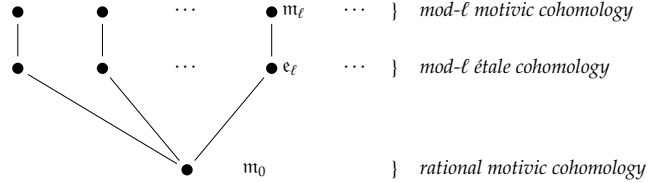
\begin{center}
  \scalebox{0.75}{
    \begin{tikzpicture}[shorten >=3pt,shorten <=3pt, level
      1/.style={level distance=1.8cm}, level 2/.style={level
        distance=1cm}, every node/.style={inner sep=2pt},
      mod/.style={circle, fill=black,}, et/.style={circle, fill=black}, root/.style={circle, fill=black}, no/.style={edge from
        parent/.style={}} ]
      \node (root) [root,label={[label distance=8mm]right: $\prrat$}] {} [grow'=up] child {node [et] {} child
        {node [mod] {}} } child {node [et] {} child {node [mod] {}} }
      child[no] {node {$\cdots$} child {node {$\cdots$}} } child {node
        [et,label=right: {$\pret{\ell}$}] {} child {node [mod,label=right: {$\prmod{\ell}$}] {}} }child[no] {node {$\cdots$} child {node {$\cdots$}} } ;
      \begin{scope}[every node/.style={right}]
        \xdef\level{root} 
        \def\rightmostnode{root-3-1} 
        \foreach \text in {rational motivic cohomology, mod-{$\ell$} \'{e}tale
          cohomology, mod-{$\ell$} motivic cohomology} { \path (\level
           -|\rightmostnode) ++(35mm,0) node{$\}$} ++(5mm,0) node
           {\text}; \xdef\level{\level-1} }
       \end{scope}
       \coordinate (br) at (0,-4mm);
     \end{tikzpicture}
}
\end{center}
\vspace{0.2cm}
Here, $\ell$ runs through all prime numbers, and the points are defined by the vanishing of the cohomology theories as indicated on the right. Moreover, the proper closed subsets are precisely the finite subsets stable under specialization.
\end{thm*}
\noindent{}
From this result we easily deduce a classification of the tt-ideals in $\dtm(\overline{\Q},\Z)$ (see \cref{tt-ideals}).

Given that motives are supposed to encode the cohomological aspects of algebraic varieties, it is of course not surprising that the tt-spectrum above contains points coming from the different cohomology theories available. But it is reassuring to find that \emph{all} points are of this form, lending some support to the belief that motives are the \emph{universal} cohomology theory. (Having said that, it would be very interesting to find non-expected points as these could hint at cohomology theories not yet discovered, or other more mysterious phenomena.)

In order to generalize from rational to integral coefficients we need to understand in particular the case of finite coefficients. The following is our result in that direction (cf. \cref{dtm-ell-tt-spectrum}).
\begin{thm*}\label{intro-dtm-ell}
  Let $\F$ be an algebraically closed field, and $\ell$ a prime number invertible in $\F$. The tt-spectrum of $\dtm(\F,\Z/\ell)$ is canonically isomorphic to the homogeneous spectrum of the polynomial ring in one variable $\Z/\ell[\beta]$.
\end{thm*}
In the proof of this theorem we use on the one hand Positselski's description of Tate motives in terms of filtered Galois representations~\cite{positselski:artin-tate-motives}, and on the other hand our study of the tt-geometry of filtered modules in~\cite{gallauer:tt-fmod}. Some of Balmer's new results in tt-geometry~\cite{balmer:surjectivity} allow us then to patch our findings together with the findings for rational coefficients by Peter~\cite{peter:spectrum-damt}, to determine the underlying set of the tt-spectrum in \cref{intro-dtm-Qbar}.

Some interesting properties of Tate motives turn out to be encoded in the tt-spectrum in a slightly subtle way. Indeed, the topology of the tt-spectrum reflects the conservativity of the $\ell$-adic realization as well as the fact that the motivic cohomology of (Tate) motives behaves like a finitely generated abelian group (see \cref{uct-intuition}). 

The theory of étale motives is closely related to the theory of motives---for example the two theories coincide with rational coefficients. We study étale motives both for their own interest, and because we will use the comparison of the two theories to bear on our understanding of Tate motives. This works because with finite coefficients the étale theory is much simpler: the Rigidity Theorem of Suslin and Voevodsky identifies étale motives with Galois representations. We will use this result together with Rost, Voevodsky and others' resolution of the Bloch-Kato conjecture to confirm this simplicity in tt-geometric terms (cf.~\cref{spec-dmet-ell}).
\begin{thm*}
  Let $\F$ be a field, and $\ell$ a prime number invertible in $\F$. Assume that $\F$ contains a primitive $\ell$th root of unity (respectively, primitive $4$th root of unity if $\ell=2$). Then the tt-spectrum of $\dmet(\F,\Z/\ell)$ has a single point.
\end{thm*}

When the tt-category is rigid, \ie every object has a strong tensor dual, then the tt-spectrum can be endowed with a natural structure sheaf turning it into a locally ringed space. This has been used by Balmer to, for example, recover a topologically noetherian scheme from its category of perfect complexes~\cite[6.3]{balmer:spectrum}. Over fields such as the algebraic numbers, where we know the underlying topological space of the tt-spectrum of Tate motives completely by \cref{intro-dtm-Qbar}, we will describe this sheaf explicitly (cf. \cref{spc-sheaf}). To do so we need to understand the category one obtains from Tate motives by inverting the image of the Bott elements $\beta:\Z/\ell(0)\to\Z/\ell(1)$ appearing in \cref{intro-dtm-ell} (it amounts to the choice of a primitive $\ell$th root of unity). It has been shown by Haesemeyer and Hornbostel in~\cite{haesemeyer-hornbostel:bott} that (under some assumptions on the base field), by inverting motives with finite coefficients with respect to the Bott element one obtains étale motives. We will upgrade this result to integral coefficients (cf.~\cref{bott-invert}).
\begin{thm*}
  Let $\F$ be a field of exponential characteristic $p$, containing all roots of unity of order coprime to $p$, and of finite $\ell$-cohomological dimension for all primes $\ell\neq p$. Then there are canonical equivalences of tt-categories
  \begin{align*}
    \DM(\F,\Z[1/p])/\langle\cone(\beta_{\ell})\mid\ell\neq p\rangle^{\oplus}&\xrightarrow{\simeq}\DMet(\F,\Z[1/p])\\
    \left(
      \dm(\F,\Z[1/p])/\langle\cone(\beta_{\ell})\mid\ell\neq p\rangle
\right)^{\natural}&\xrightarrow{\simeq}\dmet(\F,\Z[1/p]).
  \end{align*}
  The same result holds for the effective versions.
\end{thm*}

\hypersetup{bookmarksdepth=-2}
\addtocontents{toc}{\SkipTocEntry}
\section*{Acknowledgment}
\hypersetup{bookmarksdepth}
I would like to thank Simon Pepin Lehalleur for bringing~\cite{positselski:artin-tate-motives} to my attention which was the starting point for this project. Many thanks to Paul Balmer, Shane Kelly, and Simon Pepin Lehalleur for their valuable input on an earlier version of this article, and to Denis-Charles Cisinski for pointing out a flaw in an argument and the connection to the conservativity of the $\ell$-adic realizations. I'm also grateful to Joseph Ayoub for an encouraging discussion concerning \cref{sec:bott}.

\section{Conventions}
\label{sec:conventions}
Our conventions regarding tensor triangular geometry mostly follow those of~\cite{balmer:spectrum,balmer:sss}. A \emph{tensor triangulated category} (or \emph{tt-category} for short) is a triangulated category with a compatible (symmetric, unital) tensor structure. If not specified otherwise, the tensor product is denoted by $\otimes$ and the unit by~$\one$. A \emph{tt-functor} is an exact tensor functor between tt-categories.

A \emph{tt-ideal} in a tt-category $\T$ is a thick subcategory $\I\subset \T$ such that~$\T\otimes \I\subset\I$. If $S$ is a set of objects in $\T$ we denote by $\langle S\rangle$ the tt-ideal generated by~$S$. To a small tt-category $\T$ one associates a ringed space $\spec(\T)$, called the \emph{tt-spectrum of $\T$}, whose underlying topological space is denoted by $\spc(\T)$. It is a spectral space and consists of \emph{prime ideals} in $\T$, \ie proper tt-ideals $\I$ such that $a\otimes b\in\I$ implies $a\in\I$ or~$b\in\I$. A base for the closed subsets of the topology is given by the supports $\supp(a)$ of objects $a\in\T$; here, $\supp(a)=\{\mathfrak{P}\mid a\notin\mathfrak{P}\}$. The complement of $\supp(a)$ is denoted by $U(a)$. If $\T$ is rigid then $\spec(\T)$ is a \emph{locally} ringed space.\footnote{We will say a bit more about the association of $\T\mapsto\spec(\T)$ in \cref{sec:tt-sheaf}.}

All rings are commutative with unit, and morphisms of rings are unital. For $R$ a ring, we denote by $\spec(R)$ the Zariski spectrum of $R$ (considered as a locally ringed space) whereas $\spc(R)$ denotes its underlying topological space (as for the tt-spectrum). We adopt similar conventions regarding graded rings $R$: they are commutative in a general graded sense~\cite[3.4]{balmer:sss}, and possess a unit. $\spech(R)$ denotes the homogeneous (Zariski) spectrum with underlying topological space~$\spch(R)$. This differs from $\mathrm{Proj}(R)$ in that the prime ideals making up the underlying space may contain the irrelevant ideal in $R$. It is still a spectral space though.

Recall also that Balmer constructs~\cite{balmer:sss} comparison maps between the tt-spectrum and certain Zariski spectra. Explicitly, there is a canonical spectral morphism $\rho:\spec(\T)\to\spec(\R_{\T})$, where $\R_{\T}=\End_{\T}(\one)$ denotes the endomorphism ring of the unit in $\T$, called the \emph{central ring}. More generally, fixing an invertible object $u\in\T$, he considers the \emph{graded central ring} $\R^{\bullet}_{\T}=\hom_{\T}(\one,u^{\otimes\bullet})$. There is then a canonical spectral morphism $\rho^{\bullet}:\spec(\T)\to\spech(\R_{\T}^{\bullet})$, given by
\begin{equation*}
  \rho^{\bullet}(\mathfrak{P})=\{r\in R^{\bullet}_{\T}\text{ homogeneous }\mid\cone(r)\notin\mathfrak{P}\}.
\end{equation*}
The map $\rho$ is just the restriction of $\rho^{\bullet}$ to the degree 0 part. We will repeatedly use the fact that $\rho$ (respectively $\rho^{\bullet}$) is a  homeomorphism if and only if it is an isomorphism of locally ringed spaces \cite[6.11]{balmer:sss}.

\section{Triangulated category of (Tate) motives}
\label{sec:DTM}

Fix a ring $R$ and a field $\F$. In these two preliminary sections (\cref{sec:DTM} and~\ref{sec:DTMet}) we are going to recall some generalities on categories of motives over $\F$ with coefficients in $R$. Many of the constructions and proofs go back to work of Voevodsky, Suslin, Bloch, Levine, and many others. Our discussion will be too brief for some---we recommend~\cite{voevodsky00-mm,MVW-motcoh} as introductions instead. As main reference for this section we will use~\cite{cisinski-deglise:dm}.

There is a (large) tt-category of ``big'' motives $\DM(\F,R)$~\cite[11.1.1, 11.1.2]{cisinski-deglise:dm} constructed from the derived category of Nisnevich sheaves with transfers of $R$-modules on the category $\sm/\F$ of smooth (finite type, separated) $\F$-schemes. In particular, it comes with an ``associated motive'' functor
\begin{align*}
  R(-):\sm/\F&\to\DM(\F,R)\\
  X&\mapsto R(X).
\end{align*}
The tensor structure on $\DM(\F,R)$ is determined by two facts:
\begin{itemize}
\item The functor $R(-)$ is symmetric unital monoidal (also with respect to transfers), for example $R(X\times_{\spec{\F}}Y)=R(X)\otimes R(Y)$, and $R(\spec(\F))=:R(0)$ is the unit. (If the context allows we will write simply $R$ for $R(0)$.)
\item The reduced motive of $\mathbb{P}^{1}_{\F}$, denoted by $R(1)[2]$, is ($\otimes$-)invertible. One then gets for any integer $n$ an invertible $R(n)$, called the \emph{Tate twist of weight $n$}. Clearly, $R(i)\otimes R(j)=R(i+j)$, and the dual of $R(i)$ is $R(-i)$.
\end{itemize}
The triangulated category $\DM(\F,R)$ is compactly generated, and a set of compact generators is given by the motives of smooth $\F$-schemes~\cite[11.1.6]{cisinski-deglise:dm}. In fact, in a sense one can make precise, $\DM(\F,R)$ is generated by $R(X)$ ($X$ runs through smooth $\F$-schemes) and $R(-1)$, subject to Nisnevich descent and the relations $R(1)\otimes R(-1)=R(0)$, $R(\mathbb{A}^{1}_{\F})=R(0)$. We denote the subcategory of compact objects by $\dm(\F,R)$. Its objects are often called \emph{geometric motives}, or just \emph{motives} if no confusion with big motives is possible (or if the distinction in the given context is immaterial). By what was said above, $\dm(\F,R)$ is the thick subcategory generated by $R(X)(n):=R(X)\otimes R(n)$, where $X$ is a smooth $\F$-scheme and $n\in\Z$. It is a (small) idempotent complete tt-category. Moreover, it is rigid if the exponential characteristic\footnote{Recall that the exponential characteristic of $\F$ is 1 if $\car(\F)=0$, and $p$ if $\car(\F)=p>0$.} of $\F$ is invertible in $R$ [\citealp[5.3.18]{kelly:dm-ldh}; \citealp[8.1]{cisinski-deglise:integral-mixed-motives}]. (Conjecturally, rigidity also holds without inverting the exponential characteristic as shown in~\cite{voevodsky00-mm}.)

The \emph{triangulated category of Tate motives} is the thick subcategory generated by the Tate twists $R(n)$, $n\in\Z$. It is denoted by $\dtm(\F,R)$. It is a (small) rigid, idempotent complete tt-category. There is also a ``big'' version: $\DTM(\F,R)$ denotes the localizing subcategory of $\DM(\F,R)$ generated by Tate twists. It is a (large) tt-category.

Given a ring morphism $R\to R'$, one can associate to a Nisnevich sheaf with transfers of $R$-modules $F$ the sheafification of $X
\mapsto F(X)\otimes_{R}R'$, a Nisnevich sheaf with transfers of $R'$-modules. This induces an adjunction
\begin{equation*}
  \bc:\DM(\F,R)\rightleftarrows\DM(\F,R'):\ff,
\end{equation*}
the right adjoint being induced by forgetting the $R'$-structure. The functor $\ff$ is conservative~[see the proof of \citealp[A.6]{ayoub:etale-realization} or \citealp[5.4.2]{cisinski-deglise:etale-motives}], and for any motive $M\in\DM(\F,R)$, we have $\ff\bc(M)=M\otimes R'$, where $R'$ is the constant sheaf associated to $R'$, considered as an object in $\DM(\F,R)$. The tt-functor $\bc$ sends $R(X)(n)$ to $R'(X)(n)$ and therefore restricts to tt-functors
\begin{align}\label{coefficients-adjunction-compact}
  \dm(\F,R)\to\dm(\F,R'),&& \dtmDTM(\F,R)\to\dtmDTM(\F,R').
\end{align}
If $R'$ is a perfect $R$-module (\ie $R'\in\D{R}^{\mathrm{perf}}$) then the right adjoint $\ff$ also preserves compact objects.

The hom sets in the triangulated category of motives are closely related to important algebraic geometric invariants. Specifically, for a smooth $\F$-scheme $X$ and integers $m,n$, the groups
\begin{equation}\label{motivic-cohomology}
  \Hmm^{m,n}(X,R):=\hom_{\DM(\F,\Z)}(\Z(X),R(n)[m])=\hom_{\DM(\F,R)}(R(X),R(n)[m])
\end{equation}
are the \emph{motivic cohomology groups} of $X$ with coefficients in $R$. There is a canonical identification \cite{voevodsky:motcoh-CH},
\begin{equation}\label{higher-chow-groups}
  \Hmm^{m,n}(X,R)=\mathrm{CH}^{n}(X,2n-m;R),
\end{equation}
with a direct generalization of Chow groups, called Bloch's higher Chow groups, thereby linking motivic cohomology to algebraic cycles. In particular, the motivic cohomology groups of $\spec(\F)$ on the ``diagonal'' are canonically isomorphic to Milnor $K$-theory~\cite{totaro:milnor-k-theory,nesterenko-suslin:milnor-k-theory}:
\begin{equation}\label{milnor-k-theory}
  \Hmm^{n,n}(\spec(\F),R)=K_{n}^{\scriptscriptstyle\mathrm{M}}(\F)\otimes R.
\end{equation}
More generally, we define for any motive $M\in\DM(\F,R)$ and any ring $R'$ over $R$,
\begin{equation*}
  \Hmm^{m,n}(M,R'):=\hom_{\DM(\F,R)}(M,R'(n)[m])=\hom_{\DM(\F,R')}(\bc M,R'(n)[m])
\end{equation*}
the \emph{motivic cohomology groups of $M$ with coefficients in $R'$}.

\section{Triangulated category of étale (Tate) motives}
\label{sec:DTMet}

We want to discuss étale versions of the constructions in the previous section. Thus instead of Nisnevich sheaves with transfers we consider étale sheaves with transfers. Our main references for this section are~\cite{ayoub:icm,ayoub:etale-realization,cisinski-deglise:etale-motives}.

There is a (large) tt-category of ``big'' étale motives $\DMet(\F,R)$ [\citealp[\S4.1.1]{ayoub:icm}; \citealp[2.2.4]{cisinski-deglise:etale-motives}] constructed from the derived category of étale sheaves with transfers of $R$-modules on the category $\sm/\F$ of smooth (finite type, separated) $\F$-schemes. In particular, it comes again with a symmetric unital monoidal ``associated étale motive'' functor
\begin{align*}
  R^{\et}(-):\sm/\F&\to\DMet(\F,R)\\
  X&\mapsto R^{\et}(X).
\end{align*}
In contrast to the situation of the previous section, the image is in general not compact, and $\DMet(\F,R)$ is in general not compactly generated (the reason being that the étale cohomological dimension of $\F$ can be infinite). We denote the thick subcategory generated by $R^{\et}(X)(n)$ for $X$ smooth and $n$ an integer by $\dmet(\F,R)$. It is called the \emph{triangulated category of geometric étale motives}~\cite[4.3]{ayoub:icm}. It is a (small) idempotent complete tt-category. Moreover, for many coefficient rings $R$ (including any localization or quotient of $\Z$) it is rigid, by~\cite[6.3.26]{cisinski-deglise:etale-motives} and \cref{etale-motives-models} below. The \emph{triangulated category of étale Tate motives} $\dtmet(\F,R)$ is the thick subcategory generated by the Tate twists $R^{\et}(n)$, $n\in\Z$. It is a (small) rigid, idempotent complete tt-category.

\begin{rmk}\label{etale-motives-models}
  There are at least two other models for the categories just introduced. One uses étale sheaves \emph{without} transfers~\cite{ayoub:icm,ayoub:etale-realization}, the other uses $\mathrm{h}$-sheaves on the category of finite type $\F$-schemes~\cite[5.1.3]{cisinski-deglise:etale-motives}. That they indeed coincide up to canonical equivalence follows for example from~[\citealp[B.1]{ayoub:etale-realization} and~\citealp[5.5.5]{cisinski-deglise:etale-motives}]. When citing results from the literature about étale motives we will therefore freely use any of the three models.
\end{rmk}

One can define analogous \emph{étale motivic cohomology groups}, and these again are closely related to algebraic cycles~[\citealp[4.12]{ayoub:icm}; \citealp[7.1.2]{cisinski-deglise:etale-motives}]---except that one loses all $p$-torsion information if $\car(\F)=p>0$. Indeed, the categories of étale motives just introduced are all $\Z[1/p]$-linear. This follows from the existence of the Artin-Schreier sequence of étale sheaves
  \begin{equation*}
    0\to \Z/p\Z\to\mathbb{G}_{a}\xrightarrow{F_{p}-1}\mathbb{G}_{a}\to 0,
  \end{equation*}
  where $F_{p}$ denotes the Frobenius. It induces a triangle in $\dtmet(\F,R)$, and since $F_{p}-1$ induces an isomorphism on $R^{\et}=R^{\et}(\mathbb{G}_{a})$ in $\dmet(\F,R)$, multiplication by $p$ is an automorphism.

As before, a ring morphism $R\to R'$ induces tt-functors~[\citealp[A.2]{ayoub:etale-realization}; \citealp[5.4.1]{cisinski-deglise:etale-motives}]
\begin{align*}
  \bc:\dmet(\F,R)\to\dmet(\F,R'),&&\bc:\dtmet(\F,R)\to\dtmet(\F,R'),
\end{align*}
and again, if $R'$ is perfect over $R$ then they admit a right adjoint $\ff$, in which case one has $\ff\bc(M)=M\otimes R'$.

Fix a prime number $\ell$ which is invertible in $\F$, and consider the category $\Md(G_{\F},\Z/\ell)$ of discrete $G_{\F}$-modules over $\Z/\ell$, where $G_{\F}$ denotes the absolute Galois group of $\F$. Its derived category is denoted by $\D{G_{\F},\Z/\ell}$. There is an étale realization functor~[\citealp{ivorra:l-adic-real-1}; \citealp[5.2]{ayoub:etale-realization}; \citealp[7.2]{cisinski-deglise:etale-motives}]
\begin{equation*}
  \reet_{\ell}:\dmet(\F,\Z)\to\D{G_{\F},\Z/\ell}_{c}^{b}
\end{equation*}
to the subcategory of $\D{G_{\F},\Z/\ell}$ spanned by complexes with bounded, finite dimensional cohomology. This is a tt-functor which moreover factors through a fully faithful tt-functor~\cite[5.5.4]{cisinski-deglise:etale-motives}, abusively denoted by the same symbol,
\begin{equation}\label{rigidity}
  \reet_{\ell}:\dmet(\F,\Z/\ell)\to\D{G_{\F},\Z/\ell}^{b}_{c}.
\end{equation}
This is a form of the \emph{Rigidity Theorem} of Suslin and Voevodsky. \cref{rigidity} is an equivalence if $\F$ is of finite $\ell$-cohomological dimension.

Finally, there is a canonical étale sheafification tt-functor
\begin{equation}\label{etale-sheafification}
  \aet:\dm(\F,R)\to \dmet(\F,R),
\end{equation}
which takes $R(X)(n)$ to $R^{\et}(X)(n)$, restricts to a corresponding tt-functor on Tate (respectively, étale Tate) motives, and is compatible with change of coefficients. We continue to denote by $\reet_{\ell}$ the composition $\reet_{\ell}\circ\aet$ whenever this makes sense. Essentially because higher Galois cohomology is torsion, the étale sheafification induces an equivalence
\begin{equation}
  \label{rational-equivalence}
  \aet:\dm(\F,R)\xrightarrow{\sim} \dmet(\F,R)
\end{equation}
whenever $\Q\subset R$~\cite[16.1.2]{cisinski-deglise:dm}.

\section{Torsion and finite generation in motivic cohomology}
\label{sec:motivic-cohomology}

In this section we collect some basic properties concerning rational and mod-$\ell$ motivic cohomology for (Tate) motives. The discussion will culminate in the proof that the motivic cohomology groups of Tate motives behave ``as if they were finitely generated''. We refer to \cref{uct-intuition} for elaboration.

For the reader unfamiliar with the objects introduced in the previous two sections, this is a good opportunity to get better acquainted with the formalism. Throughout the section, $\F$ is an arbitrary field.

We start with a result characteristic of Tate motives.
\begin{lem}\label{cellularity-property}
  Let $M\in\dtm(\F,R)$ be a Tate motive. The following are equivalent:
  \begin{enumerate}
  \item $M=0$,
  \item $\Hmm^{\bullet,\bullet}(M,R)=0$.
  \end{enumerate}
\end{lem}
\begin{proof}
  If $M$ has trivial motivic cohomology groups $\Hmm^{m,n}(M,R)$ for all integers $m,n$, this means that for the dual $M^{\vee}$ of $M$, we have
  \begin{equation*}
    \hom_{\dtm(\F,R)}(R(-n)[-m],M^{\vee})=\hom_{\dtm(\F,R)}(M,R(n)[m])=\Hmm^{m,n}(M,R)=0
  \end{equation*}
  for all integers $m,n$. Since the set of Tate twists generates $\dtm(\F,R)$ we see that $M^{\vee}=0$ and this implies $M=0$.
\end{proof}

\begin{con}
  From now on and until the end of the section we assume that $R$ is a localization of $\Z$ such that $\Z\subset R\subset\Q$.
\end{con}

Let $F:\dtm(\F,R)\to\T$ be a tt-functor, and $M\in\dtm(\F,R)$ a Tate motive. We say that $M$ is \emph{$F$-acyclic} if $F(M)=0$. The $F$-acyclic objects clearly define a tt-ideal in $\dtm(\F,R)$.

\begin{exa}\label{etale-cohomology}
  Suppose $\ell$ is a prime number invertible in $\F$. Consider the tt-functor $\reet_{\ell}:\dtm(\F,\Z)\to\D{G_{\F},\Z/\ell}_{c}^{b}$. The $\reet_{\ell}$-acyclic objects are those whose mod-$\ell$ étale cohomology vanishes.
\end{exa}

 If $A\in\DM(\F,R)$ is any motive, we say that $M$ is \emph{$A$-acyclic} if it is acyclic with respect to the functor $-\otimes A:\dtm(\F,R)\to\DM(\F,R)$. We will be most interested in $A=\Z/\ell$ ($\ell$ a prime) or $A=\Q$.

\begin{lem}\label{Q-localization}
 Let $S$ be the multiplicative subset $\Z\backslash 0$. In each case of motives, Tate motives, étale motives, or étale Tate motives, we have canonical equivalences of tt-categories
  \begin{equation*}
    \left(
      S^{-1}\dmgen(\F,R)
\right)^{\natural}=
\left(
  \dmgen(\F,R)\otimes\Q
\right)^{\natural}=\dmgen(\F,\Q),
  \end{equation*}
  where $(-)^{\natural}$ denotes the idempotent completion.
\end{lem}
\begin{proof}
  The categorical (or Verdier) localization at $S$ is the naive localization (\ie the category obtained by localizing each hom set), by~[\citealp[9.1]{ayoub:etale-realization}; \citealp[3.6]{balmer:sss}]. This gives the first equivalence. For the second equivalence, the étale version is~\cite[5.4.9]{cisinski-deglise:etale-motives}; the non-étale version is simpler, and can be found in~\cite[11.1.5]{cisinski-deglise:dm}.
\end{proof}

\begin{cor}\label{torsion}
  Let $M\in\dtm(\F,R)$ be a Tate motive. The following are equivalent:
  \begin{enumerate}
  \item $\Hmm^{\bullet,\bullet}(M,\Q)=0$.
  \item $M$ is $\Q$-acyclic.
  \item $M$ is $n$-torsion for some positive integer $n$, \ie $n\cdot\id_{M}=0$.
  \end{enumerate}  
\end{cor}
\begin{proof}
  By the change of coefficients adjunction for $R\to\Q$, we have
  \begin{equation*}
    \hom_{\dtm(\F,\Q)}(\bc M,\Q(n)[m])=\Hmm^{m,n}(M,\Q),
  \end{equation*}
  and by \cref{cellularity-property} these groups all vanish if and only if $\bc M=0$. But $\ff$ is conservative (\cref{sec:DTM}) so this is equivalent to
  \begin{equation*}
    0=\ff\bc M=M\otimes\Q,
  \end{equation*}
  \ie it is equivalent to $M$ being $\Q$-acyclic. \cref{Q-localization} shows that the second and third condition are equivalent as well.
\end{proof}

The first part of the proof also gives the following result.
\begin{lem}\label{mod-ell-cohomology}
  Let $M\in\dtm(\F,R)$ be a Tate motive, and $\ell$ a prime number. The following are equivalent:
  \begin{enumerate}
  \item $\Hmm^{\bullet,\bullet}(M,R/\ell)=0$.
  \item $M$ is $R/\ell$-acyclic.
  \end{enumerate}
\end{lem}

\begin{rmk}\label{uct-intuition}
  Among finitely generated abelian groups there is a stark divide between finite and infinite ones. This simple observation is used in many arguments in algebraic topology which involve the universal coefficient theorem. More explicitly, if $X$ is a topological space with finitely generated cohomology groups $\Hm^{\bullet}(X,\Z)$ then there is the following dichotomy:
  \begin{itemize}
  \item If $\Hm^{m}(X,\Q)=0$ then for almost all primes $\ell$, $\Hm^{m}(X,\Z/\ell)=0$.
  \item If $\Hm^{m}(X,\Q)\neq 0$ then for all primes $\ell$, $\Hm^{m}(X,\Z/\ell)\neq 0$.
  \end{itemize}
Now, it is of course not true that the motivic cohomology groups of Tate motives are finitely generated in general. (For example, \cref{milnor-k-theory} shows that $\Hmm^{1,1}(\F,\Z)=\F^{\times}$.) However, we would like to establish that they still exhibit a similar behavior. We can prove the first part of the dichotomy here, as a consequence of the preceding results in this section. The second part of the dichotomy will not be stated here, but is also true (for $\ell\in\F^{\times}$). It will be seen to follow from the conservativity of the $\ell$-adic realization for Tate motives (\cf the proof of \cref{spec-dtmet}).

This observation will have important ramifications for the topology of the tt-spectrum of Tate motives.
\end{rmk}

\begin{pro}\label{uct}
  Let $M\in\dtm(\F,R)$ be a Tate motive and assume $\Hmm^{\bullet,\bullet}(M,\Q)=0$. Then for almost all primes $\ell$ we have $\Hmm^{\bullet,\bullet}(M,R/\ell)=0$.
\end{pro}
\begin{proof}
  Assume that $M\in\dtm(\F,R)$ has trivial rational motivic cohomology, \ie $M$ is $n$-torsion for some positive integer $n$, by \cref{torsion}. Let $\ell$ be any prime not dividing $n$. Consider the object $\bc M\in\dtm(\F,R/\ell)$, where $\bc$ is the change of coefficients functor associated to $R\to R/\ell$. Multiplication by $n$  on $\bc M$ is at the same time zero (since it is so on $M$), and an isomorphism (since the category $\dtm(\F,R/\ell)$ is $R/\ell$-linear). We conclude that $\bc M=0$ as claimed, cf.~\cref{mod-ell-cohomology}.
\end{proof}

\begin{rmk}
  We have phrased most of the results in this section so far for Tate
  motives but it should be remarked that they equally hold for étale
  Tate motives with the same proofs, taking into account the following observations:
  \begin{itemize}
  \item Rational étale motivic cohomology is the same thing as rational motivic cohomology as discussed above (\cref{rational-equivalence}).
  \item By Rigidity (\cref{rigidity}), mod-$\ell$ étale motivic cohomology is the same thing as mod-$\ell$ étale cohomology.
  \item One sometimes has to modify the statements and arguments to take into account that the category of étale motives is $\Z[1/p]$-linear, where $p$ is the exponential characteristic of $\F$.
  \end{itemize}

Moreover, although we are mainly interested in Tate motives in this article, it should be said that some of the arguments in this section apply to motives in general. In particular, we have just shown that an implication similar to the one in \cref{uct} holds in $\dm(\F,R)$ or $\dmet(\F,R)$.
\end{rmk}

\begin{pro}
  Let $M\in\dm(\F,R)$ (or $\dmet(\F,R)$) be a motive, and assume $M\otimes\Q=0$. Then for almost all primes $\ell$ we have $M\otimes R/\ell=0$.
\end{pro}
\section{tt-geometry of étale (Tate) motives}
\label{sec:etale}
The tt-geometry of étale Tate motives is easier to describe than the one of Tate motives, mostly due to our good understanding of étale motives with finite coefficients (the Rigidity Theorem, see \cref{rigidity}). In fact, with finite coefficients it is not more difficult to describe the tt-geometry of étale motives (not necessarily Tate) at the same time, which is what we are going to start with. The integral versions considered subsequently require the (rational) motivic cohomology of the base field to satisfy certain vanishing conditions, and are proved for Tate motives only.

Fix a field $\F$ and a prime number $\ell$ invertible in $\F$. We want to recall the Bloch-Kato conjecture (now a theorem). The short exact sequence of étale sheaves
\begin{equation*}
  0\to\mu_{\ell}\to\mathcal{O}^{\times}\xrightarrow{\cdot\ell}\mathcal{O}^{\times}\to 0
\end{equation*}
induces a canonical map $\F^{\times}\to \Hmet^{1}(\F,\mu_{\ell})$ in étale cohomology. Using the cup product this extends to a morphism of graded rings $K^{\scriptscriptstyle\mathrm{M}}_{\bullet}(\F)\to\Hmet^{\bullet}(\F,\mu_{\ell}^{\otimes\bullet})$ which clearly annihilates $\ell$. Voevodsky, Rost, and others, show that the induced map is an isomorphism (as Bloch and Kato conjectured):
  \begin{equation*}
    K^{\scriptscriptstyle\mathrm{M}}_{\bullet}(\F)/\ell\xrightarrow{\sim}\Hmet^{\bullet}(\F,\mu_{\ell}^{\otimes\bullet})=\Hm^{\bullet}(G_{\F},\mu_{\ell}(\overline{\F})^{\otimes\bullet}).
  \end{equation*}

In order to apply this result we need to know the homogeneous spectrum of the graded rings involved.
\begin{lem}\label{spec-milnor}
  Let $\F$ be a field, and $\ell$ a prime. If $\ell=2$ we assume that $-1$ is a sum of squares in $\F$. Then the graded ring $K^{\scriptscriptstyle\mathrm{M}}_{\bullet}(\F)/\ell$ has a unique homogeneous prime ideal (namely, $K^{\scriptscriptstyle\mathrm{M}}_{>0}(\F)/\ell$).
\end{lem}
\begin{proof}
  We distinguish two cases:
  \begin{itemize}
  \item[$\ell=2$:] We are assuming that $-1$ is a sum of squares in $\F$. This is equivalent, by~\cite[1.4]{milnor:k-theory-quadratic}, to every element in $K^{\scriptscriptstyle\mathrm{M}}_{\bullet}(\F)$ of positive degree being nilpotent. In particular, the only homogeneous prime ideal containing $2$ is $\langle K_{1}^{\scriptscriptstyle\mathrm{M}}(\F),2\rangle$.
  \item[$\ell\neq 2$:] For the reader's convenience we reproduce the argument in~\cite[3.9]{thornton:spech-milnor-witt}. Let $\mathfrak{p}$ be a homogeneous prime in $K^{\scriptscriptstyle\mathrm{M}}_{\bullet}(\F)$ which does not contain $2$. Let us write $[a]$ for the symbol in $K_{1}^{\scriptscriptstyle\mathrm{M}}(\F)$ associated to $a\in\F^{\times}$. Since $2[-1]=[(-1)^{2}]=[1]=0$ we have $[-1]\in\mathfrak{p}$. But then for every $a\in \F^{\times}$, $[a]^{2}=[a][-1]\in\mathfrak{p}$~\cite[1.2]{milnor:k-theory-quadratic}, and so, again, we find $K_{1}^{\scriptscriptstyle\mathrm{M}}(\F)\subset\mathfrak{p}$.
  \end{itemize}
\end{proof}

\begin{thm} \label{spec-dmet-ell}
  Let $\F$ be a field, and let $\ell$ be a prime invertible in $\F$. Assume that $\F$ contains a primitive $\ell$th root of unity (respectively, primitive $4$th root if $\ell=2$). Then both canonical morphisms in the following composition are isomorphisms of locally ringed spaces:
  \begin{equation*}
    \spec(\dmet(\F,\Z/\ell))\to\spec(\dtmet(\F,\Z/\ell))\xrightarrow{\rho}\spec(\Z/\ell).
  \end{equation*}
  
\end{thm}
\begin{proof} The proof will proceed in several steps.
  \begin{enumerate}
  \item Note that it suffices to prove that $\spec(\dmet(\F,\Z/\ell))$ has at most one point. Since $\dmet(\F,\Z/\ell)$ is not the trivial category, its spectrum then has exactly one point, and therefore the composition $\rho:\spec(\dmet(\F,\Z/\ell))\to\spec(\Z/\ell)$ is a homeomorphism. Since the inclusion $\dtmet(\F,\Z/\ell)\to\dmet(\F,\Z/\ell)$ is (fully) faithful, the induced morphism on tt-spectra is surjective~\cite[1.8]{balmer:surjectivity}. It follows that if the composition is a homeomorphism then so are both maps in the statement. But $\rho$ being a homeomorphism already implies that it is an isomorphism of locally ringed spaces.
  \item By Rigidity, $\dmet(\F,\Z/\ell)$ embeds fully faithfully into $\D{G_{\F},\Z/\ell}_{c}^{b}$ thus again a surjective map on tt-spectra
    \begin{equation*}
      \spc(\D{G_{\F},\Z/\ell}_{c}^{b})\to \spc(\dmet(\F,\Z/\ell)),
    \end{equation*}
    and we reduce to prove that the former is a singleton set.
  \item By the combination of \cref{finite-to-constructible-spectra} and \cref{spec-profinite} below, the morphism
    \begin{equation*}
      \rho^{\bullet}:\spc(\D{G_{\F},\Z/\ell}^{b}_{c})\to\spch(\Hm^{\bullet}(G_{\F},\Z/\ell))
    \end{equation*}
    is an injection. Since we are assuming that $\Z/\ell\cong\mu_{\ell}$, and by Bloch-Kato, the target of this map is the homogeneous spectrum of Milnor $K$-theory $K^{\scriptscriptstyle\mathrm{M}}_{\bullet}(\F)/\ell$. But we found in \cref{spec-milnor} that the latter has indeed just a single point.
  \end{enumerate}
\end{proof}

To complete the proof we need to compare the tt-spectrum of $\D{G,k}_{c}^{b}$, for a (discrete) field $k$ and a profinite group $G$, to the homogeneous spectrum of the cohomology ring of $G$. This will be done in two steps. Consider the category $\mdfin(G,k)$ of \emph{finite dimensional} discrete $G$-modules over $k$, and its bounded derived category $\D{\mdfin(G,k)}^{b}$. There is a canonical tt-functor
\begin{equation}\label{finite-to-constructible-functor}
  \iota:\D{\mdfin(G,k)}^{b}\to\D{G,k}_{c}^{b}.
\end{equation}

\begin{lem}\label{finite-to-constructible-spectra}
  The functor $\iota$ of \cref{finite-to-constructible-functor} is an equivalence and therefore induces a homeomorphism of tt-spectra:
  \begin{equation*}
    \spc(\iota):\spc(\D{G,k}_{c}^{b})\xrightarrow{\sim}\spc(\D{\mdfin(G,k)}^{b}).
  \end{equation*}
\end{lem}
\begin{proof}
  Let $M$ be a finite-dimensional discrete $G$-module over $k$, let $N$ be an arbitrary one, and let $f:N\to M$ be an epimorphism. By the definition of discrete modules there exists a finite quotient $G/H$ through which $G$ acts on $N$. Choose representatives $g_{1},\ldots{},g_{r}\in G$ for this quotient, and choose lifts $n_{1},\ldots{},n_{s}\in N$ of a $k$-basis of $M$. It is then clear that the $k$-linear hull of
  \begin{equation*}
    \{g_{i}n_{j}\mid 1\leq i\leq r, 1\leq j\leq s\}\subset N
  \end{equation*}
  is a $G$-submodule of $N$ which is in addition finite-dimensional and still surjects onto $M$. This is enough to deduce that the functor $\D{\mdfin(G,k)}^{b}\to \D{G,k}^{b}$ is fully faithful \cite[12.1]{keller:derived-categories}. The image of this functor consists of those complexes with finite-dimensional cohomology. Indeed, the latter subcategory is generated, as a triangulated subcategory, by complexes with finite-dimensional cohomology concentrated in a single degree; and these are clearly in the image of the functor.
\end{proof}

Let us now consider the graded central ring $\R^{\bullet}_{G}$ in $\D{\mdfin(G,k)}^{b}$ with respect to $k[1]$. The following statement generalizes the analogous
 result for finite groups which was proved in~\cite[8.5]{balmer:sss}, completing work of many others. Our proof will consist in reducing to the finite case and is therefore not independent.
\begin{pro}\label{spec-profinite} Let $G$ be a profinite group, and $k$ a field. Then:
  \begin{enumerate}
  \item $\R^{\bullet}_{G}$ is canonically isomorphic to $\Hm^{\bullet}(G,k)$.
  \item The comparison morphism
  \begin{equation*}
    \rho^{\bullet}:\spec(\D{\mdfin(G,k)}^{b})\to\spech(\Hm^{\bullet}(G,k))
  \end{equation*}
  is an isomorphism of locally ringed spaces.
  \end{enumerate}
\end{pro}
\begin{proof}
The group $G$ is the inverse limit $\varprojlim_{i}G_{i}$ of a cofiltered diagram $I\ni i\mapsto G_{i}$ of finite groups (with the discrete topology). For every transition map $G_{i}\to G_{j}$ we obtain a functor $\mdfin(G_{j},k)\to\mdfin(G_{i},k)$ by restricting the action. Notice that a finite dimensional $G$-module is discrete (\ie the action of $G$ is continuous) if and only if the action factors through $G_{i}$, for some $i$. It follows easily that $\varinjlim_{i}\mdfin(G_{i},k)=\mdfin(G,k)$ and this equivalence passes first to the level of bounded cochain complexes, and then to the bounded derived category so that we have $\varinjlim_{i}\D{\mdfin(G_{i},k)}^{b}=\D{\mdfin(G,k)}^{b}$.

We now obtain the first part of the statement since $\R_{G_{i}}^{\bullet}=\Hm^{\bullet}(G_{i},k)$ and therefore
\begin{equation*}
  \R_{G}^{\bullet}=\varinjlim\R_{G_{i}}^{\bullet}=\varinjlim\Hm^{\bullet}(G_{i},k)=\Hm^{\bullet}(G,k).
\end{equation*}

For the second statement, consider the following square which is commutative by naturality of $\rho^{\bullet}$ [\citealp[5.6]{balmer:sss}; recalled in \cref{sec:tt-sheaf}].
\begin{equation*}
  \xymatrix{\spc(\D{\mdfin(G,k)}^{b})\ar[r]\ar[d]_{\rho^{\bullet}}&\varprojlim_{i}\spc(\D{\mdfin(G_{i},k)}^{b})\ar[d]^{\varprojlim_{i}\rho^{\bullet}}\\
    \spch(\Hm^{\bullet}(G,k))\ar[r]&\varprojlim_{i}\spch(\Hm^{\bullet}(G_{i},k))}
\end{equation*}
By~\cite[8.2]{gallauer:tt-fmod}, the top horizontal arrow is a homeomorphism. By~\cite[8.5]{balmer:sss}, so is the right vertical map. And, again by $\varinjlim\Hm^{\bullet}(G_{i},k)=\Hm^{\bullet}(G,k)$, the bottom horizontal map is a homeomorphism. Necessarily then, the left vertical map is a homeomorphism as well. We now conclude since $\rho^{\bullet}$ is then automatically an isomorphism of locally ringed spaces.
\end{proof}

Having described the tt-geometry for finite coefficients, we next consider rational coefficients. Here, we restrict to Tate motives in order to invoke the results of~\cite{peter:spectrum-damt}. Recall also (\cref{rational-equivalence}) that we have an equivalence of tt-categories $\dtm(\F,\Q)\simeq\dtmet(\F,\Q)$ and the discussion therefore applies to both topologies; we will phrase them for $\dtm(\F,\Q)$.

Consider $\Hmm^{m,n}(\spec(\F),\Q)$, the rational motivic cohomology ring of the field (see \cref{motivic-cohomology}). There is the following relation between Bloch's higher Chow groups and algebraic $K$-theory (cf.~\cref{higher-chow-groups}):
\begin{equation*}
    \Hmm^{m,n}(\spec(\F),\Q)=\mathrm{CH}^{n}(\F,2n-m;\Q)=(K_{2n-m}(\F)\otimes\Q)^{(n)},
\end{equation*}
where the latter denotes the weight $n$ eigenspace of the Adams operations $(\psi^{k})_{k}$~\cite{bloch:moving-lemma,levine:higher-chow-groups}. It might be helpful to visualize this bigraded ring as in \cref{fig:motcohpoint}, regarding which we offer a few comments:
\begin{figure}[t]
 
\begin{center}
  \begin{tikzpicture}[scale=0.7]
    \begin{axis}[axis lines=middle,xmin=-5,xmax=10, xtick={2}, ytick={2},
      ymin=-2,ymax=10,xlabel={$m$},x label style={at={(1.08,0.13)}},ylabel={$n$}, y label style={at={(0.3,1.08)}},
      after end axis/.code={%
        \begin{scope}[on background layer]
          \node [above right] at (axis cs:  10,  10) {$K^{\mathrm{M}}_{\bullet}$};
          \foreach \x in {-5,...,10}{ \foreach \y in {-2,...,10}{
              \fill[black!80] (axis cs:\x,\y) circle[radius=0.5pt];
            }}
        \end{scope} }]

      \addplot[mark=none, black, ultra thick] coordinates{(0,0) (10,10)};
      \addplot[fill=black, fill opacity=0.4, draw opacity=0] coordinates { (0.5,-0.5) (0.5,0) (10.5,10) (10.5,-0.5)} ;
      \addplot[fill=black, fill opacity=0.4, draw opacity=0] coordinates {(-5.5, -0.5) (-5.5,-2) (10.5,-2) (10.5,-0.5) };
      \addplot[fill=black, fill opacity=0.4, draw opacity=0] coordinates { (-5.5, -0.5) (-5.5,0.5) (-0.5,0.5) (-0.5,-0.5) };
      \addplot[fill=black, fill opacity=0.1, draw opacity=0] coordinates {(-5.5, 0.5) (-5.5,11) (0.5,11) (0.5,1) (0,0.5)};
      \addplot[fill=black, fill opacity=0.1, draw opacity=0] coordinates {(1.5, 1) (1.5,10) (10,10) (10,9.5)};
    \end{axis}
  \end{tikzpicture}
\end{center}

  \caption{Vanishing in $\Hmm^{m,n}(\spec(\F),\Q)$}
  \label{fig:motcohpoint}
\end{figure}
\begin{itemize}
\item The Milnor $K$-theory is displayed on the diagonal for visual aid.
\item The darker area always vanishes: for $n<0$ or $n<m$ this is for dimension reasons; for $n=0$ and $m<0$ this follows from a simple computation with higher Chow groups.
\item The lightly shaded area is what we are now interested in: the \emph{Beilinson-Soulé vanishing conjecture} predicts that the left patch vanishes. It is known in a few cases, for example if $\F$ is any of the following: a finite field, a global field (in any characteristic), a function field of a genus 0 curve over a number field. We now introduce an even stronger Hypothesis.
\end{itemize}
\begin{hyp}[Vanishing Hypothesis on $\F$] \label{vanishing-hypothesis}
  The rational motivic cohomology groups $\Hmm^{m,n}(\spec(\F),\Q)$ vanish whenever
  \begin{itemize}
  \item $m\leq 0< n$, or
  \item $n\geq m\geq 2$.
  \end{itemize}
\end{hyp}

\begin{rmk}\label{vanishing-hypothesis-union}
If $\F$ is the union of subfields all of which satisfy \cref{vanishing-hypothesis} then so does $\F$. This can be seen as a very special case of the ``continuity'' of the assignment $\dm(-,\Q)$ discussed in~\cite[4.3.3]{cisinski-deglise:dm}.
\end{rmk}

\begin{rmk}
  Using \cref{vanishing-hypothesis-union}, classical computations in algebraic K-theory imply that \cref{vanishing-hypothesis} is notably satisfied in the following cases:
  \begin{enumerate}
  \item If $\F$ is an algebraic extension of $\Q$ this follows from Borel's computation of the algebraic
      $K$-theory of number fields.
    \item If $\F$ is an algebraic extension of a finite field this follows from Quillen's computation of the algebraic $K$-theory of finite fields.
    \item If $\F$ is an algebraic extension of a global field of positive characteristic this follows from Harder's Theorem \cite{harder:arithmetic-groups-function-fields}.
  \end{enumerate}
\end{rmk}

We recall the following result on Tate motives with rational coefficients.
\begin{pro}[{\cite[4.17]{peter:spectrum-damt}}] \label{spec-dtm-rational}
Let $\F$ be a field satisfying \cref{vanishing-hypothesis}. Then both morphisms in the following composition are isomorphisms of locally ringed spaces:
  \begin{equation*}
    \spec(\dtmet(\F,\Q))\xrightarrow{\spec(\aet)}\spec(\dtm(\F,\Q))\xrightarrow{\rho}\spec(\Q).
  \end{equation*}  
\end{pro}
 The Beilinson-Soulé part of \cref{vanishing-hypothesis} is used to invoke~\cite{levine92-tatemotives} and obtain a bounded t-structure on $\dtm(\F,\Q)$. The second part of \cref{vanishing-hypothesis} then implies that the tt-spectrum is identified with the ``coherent spectrum'' of the heart~\cite[4.2]{peter:spectrum-damt}, which is easily seen to be a singleton.

Now we can put the results on finite and rational coefficients together to arrive at an integral statement (for more general statements see \cref{spec-dtmet-general} and \cref{etale-artin-tate}).
\begin{thm}\label{spec-dtmet}
  Let $\F$ be a field of exponential characteristic $p$, and assume that for every $\ell\neq p$ prime, $\F$ contains a primitive $\ell$th root of unity (respectively, $4$th root of unity if $\ell=2$). Assume also that $\F$ satisfies \cref{vanishing-hypothesis}. Then:
\begin{enumerate}
  \item The central ring $\R_{\dtmet(\F,\Z)}$ is $\Z[1/p]$.
  \item The comparison morphism
    \begin{equation*}
      \rho:\spec(\dtmet(\F,\Z))\to\spec(\Z[1/p])
    \end{equation*}
    is an isomorphism of locally ringed spaces.
\end{enumerate}
\end{thm}
\begin{proof}
  We discussed in \cref{sec:DTMet} why $p$ is a unit in the central ring. A more precise statement is~\cite[7.1.2]{cisinski-deglise:etale-motives} from which the first part of the Theorem follows immediately.

  For the second part, it suffices to show that $\rho$ is a homeomorphism. First we show it is a bijection. For $\ell\neq p$, the fiber of $\rho$ over $\langle\ell\rangle$ is by definition $\supp(\Z/\ell)$. Since $\Z/\ell$ is the image of the unit under $\ff$, the right adjoint in the adjunction
  \begin{align*}
      \bc:\dtmet(\F,\Z)\rightleftarrows\dtmet(\F,\Z/\ell):\ff,
  \end{align*}
  we deduce from~\cite[1.7]{balmer:surjectivity} that the fiber of $\rho$ over $\langle\ell\rangle$ is precisely the image of $\spc(\bc)$. And by \cref{spec-dmet-ell} this consists of a single point.

  For the generic fiber of $\rho$ we note that the central localization at $\Z\backslash 0$ gives, up to idempotent completion, exactly $\dtmet(\F,\Q)$ (\cref{Q-localization}). By~\cite[5.6]{balmer:sss} (recalled in \cref{spec-localization-cartesian-explicit}), it suffices to show that the latter has a one point spectrum. This is \cref{spec-dtm-rational}.

  Finally, we need to show that the topologies of the two spaces are the same. We use that a bijective spectral map between spectral spaces is a homeomorphism if specializations lift along the map, which allows us to work in $\dtmet(\F,\Z_{\langle{\ell}\rangle})$ (for every $\ell$ different from $p$). In other words, we need to show the inclusion of prime ideals
  \begin{displaymath}
    \ker(-\otimes\Z/\ell)\subset \ker(-\otimes\Q)
  \end{displaymath}
  in $\dtmet(\F,\Z_{\langle{\ell}\rangle})$. But suppose $M\otimes\Z/\ell=0$ for some $M\in\dtmet(\F,\Z_{\langle{\ell}\rangle})$. This implies that the $\ell$-adic realization of $M$ is zero, as is plain from the identification of this realization with $\ell$-adic completion in~\cite[7.2.24]{cisinski-deglise:etale-motives}. By conservativity of the $\ell$-adic realization (see~\cref{spec-dtm-rational}) we deduce that $M\otimes \Q=0$ as well, and this concludes the proof.
\end{proof}

\begin{rmk}\label{spec-dtmet-general}
  Of course, if $S\subset \Z\backslash\{0\}$ is a saturated multiplicative subset containing the exponential characteristic of $\F$ and such that for each prime $\ell\notin S$, $\F$ contains a primitive $\ell$th root of unity (respectively, $4$th root of unity if $\ell=2$) then the same arguments show (still assuming $\F$ satisfies \cref{vanishing-hypothesis}):
  \begin{enumerate}
  \item  The central ring $\R_{\dtmet(\F,\Z[S^{-1}])}$ is $\Z[S^{-1}]$.
  \item The comparison morphism
    \begin{equation*}
      \rho:\spec(\dtmet(\F,\Z[S^{-1}]))\to\spec(\Z[S^{-1}])
    \end{equation*}
    is an isomorphism of locally ringed spaces.
  \end{enumerate}
\end{rmk}

\begin{exa}\label{cyclotomic-field}
  Let $\F=\Q(\zeta_{\ell})$ be the $\ell$th cyclotomic field, where $\ell$ is an odd prime number. Then, canonically, $\spec(\dtmet(\Q(\zeta_{\ell}),\Z_{\langle\ell\rangle}))=\spec(\Z_{\langle\ell\rangle})$. It follows \cite[4.10]{balmer:spectrum} that the thick tensor ideals of $\dtmet(\Q(\zeta_{\ell}),\Z_{\langle\ell\rangle})$ are exactly
  \begin{align*}
    0,&&\{M\mid M\text{ is torsion}\},&&\dtmet(\Q(\zeta_{\ell}),\Z_{\langle\ell\rangle}).
  \end{align*}
  Notice in particular that étale cohomology $\dtmet(\Q(\zeta_{\ell}),\Z_{\langle\ell\rangle})\to\D{\Z/\ell}^{b}$ is conservative (as its kernel is a tt-ideal).
\end{exa}

\begin{rmk}\label{etale-artin-tate}
  A statement analogous to \cref{spec-dtmet} (or \cref{spec-dtmet-general}) holds for étale Artin-Tate motives if, in addition, \cref{vanishing-hypothesis} is satisfied for every finite extension $\F'/\F$. Indeed, \cref{spec-dmet-ell} clearly applies for finite coefficients; and for rational coefficients, \cite[4.17]{peter:spectrum-damt} gives the required result.
\end{rmk}

\section{Filtered Galois representations}
\label{sec:positselski}

In this section we recall Positselski's approach to describing Tate motives with finite coefficients in~\cite{positselski:artin-tate-motives}. The upshot is \cref{dtm-tt-equivalence} which tells us that the tt-geometry of Tate motives over certain fields and with finite coefficients is the same as the tt-geometry of Galois representations with a ``unipotent filtration'', the latter being seemingly more tractable. Although in this article we will eventually deal with algebraically closed fields only, it does not cause any difficulties to treat the general case here.

\begin{con}
  Throughout this section we fix a field $\F$ and a prime $\ell$ invertible in $\F$. The absolute Galois group of $\F$ is denoted by $G_{\F}$ as before. We also assume that $\F$ contains a primitive $\ell$th root of unity $\zeta\in\mu_{\ell}(\F)$ which we interpret as a morphism
  $\beta:\Z/\ell(0)\to\Z/\ell(1)$ in $\dtm(\F,\Z/\ell)$, in view of
  \begin{equation*}
    \hom_{\dtm(\F,\Z/\ell)}(\Z/\ell,\Z/\ell(1))=\mu_{\ell}(\F).
  \end{equation*}
  To see this identification, notice that by the change of coefficients adjunctions (\cref{coefficients-adjunction-compact}) we are supposed to identify $\hom_{\dtm(\F,\Z)}(\Z,\Z/\ell(1))$ with the $\ell$th roots of unity. This follows readily from identifying $\Z(1)$ with the complex $\mathcal{O}^{\times}[-1]$ of sheaves with transfers \cite[4.9]{MVW-motcoh}.
\end{con}
For any integer $n$, define the replete triangulated subcategories $\dtm(\F,\Z/\ell)^{\geq n}$ (respectively $\dtm(\F,\Z/\ell)^{<n}$) generated by $(\Z/\ell)(n')$ for $n'\geq n$ (respectively $n'< n$). The vanishing of motivic cohomology $\Hm^{p,q}(\F,\Z/\ell)$ when $q<0$ (cf.~our remarks regarding \cref{fig:motcohpoint}), implies that the pair
\begin{equation*}
  (\dtm(\F,\Z/\ell)^{\geq n},\dtm(\F,\Z/\ell)^{<n})
\end{equation*}
defines a t-structure on $\dtm(\F,\Z/\ell)$. Indeed, the arguments of~\cite[1.2]{levine92-tatemotives} go through word for word. In particular, this gives rise to adjoints $W^{\geq n},W^{<n}$ to the inclusions of these subcategories, and one deduces that every object $M\in\dtm(\F,\Z/\ell)$ admits a functorial descending filtration, called the \emph{weight filtration},
\begin{equation}\label{weight-filtration}
  0\to W^{\geq n_{0}}M\to W^{\geq n_{0}-1}M\to\cdots\to W^{\geq n_{1}}M\to M,
\end{equation}
such that the associated graded piece $\mathrm{gr}^{n}M:=W^{\geq n}W^{<n+1}M$ is a finite direct sum of copies of shifts of $\Z/\ell(n)$ (as $\Hmm^{p,q}(\F,\Z/\ell)$ vanishes when $q=0$ and $p\neq 0$).

Define the subcategory $\fil(\F,\Z/\ell)\subset \dtm(\F,\Z/\ell)$ as the smallest full subcategory containing $\Z/\ell(n)$ for all integers $n$, and closed under extensions. (The latter condition means that for any triangle in which the outer two terms lie in the subcategory, the middle one does as well.) The Bloch-Kato conjecture (recalled in \cref{sec:etale}) implies the Beilinson-Lichtenbaum conjecture~\cite{suslin-voevodsky:bloch-kato-beilinson-lichtenbaum,geisser-levine:bloch-kato-beilinson-lichtenbaum}, \ie the étale realization functor induces identifications:
\begin{align*}
  \Hmm^{p,q}(\F,\Z/\ell)=
  \begin{cases}
    \Hm^{p}(G_{\F},\mu_{\ell}(\overline{\F})^{\otimes q})&:p\leq q\\
    0&:p>q
  \end{cases}
\end{align*}
For one, this implies that there are no negative Ext groups between the Tate objects $\Z/\ell(n)$, and from this one deduces that $\fil(\F,\Z/\ell)$ is in fact an exact subcategory of $\dtm(\F,\Z/\ell)$ (\ie the triangles of $\dtm(\F,\Z/\ell)$ lying in $\fil(\F,\Z/\ell)$ define an exact structure, see~\cite{dyer:exact-triangulated}). And secondly, the Beilinson-Lichtenbaum conjecture essentially implies the following result.

\begin{pro}[cf.~{\cite[3.1]{positselski:artin-tate-motives}}]\label{etale-realization-filtered}
  The étale realization induces an equivalence of exact tensor categories
  \begin{equation*}
    \underline{\reet}_{\ell,\zeta}:\fil(\F,\Z/\ell)\xrightarrow{\sim}\filgalu(G_{\F},\Z/\ell),
  \end{equation*}
  where the latter denotes the category of (discrete) Galois modules over $\Z/\ell$ equipped with a \emph{unipotent filtration}, \ie a finite decreasing filtration whose graded pieces are finite direct sums of copies of the trivial Galois module $\Z/\ell$.
\end{pro}
A few words about the latter category. The tensor product $a\otimes b$ of two filtered Galois modules $a$ and $b$ has an induced filtration given by
\begin{equation*}
  (a\otimes b)^{n}=\Sigma_{p+q= n}a^{p}\otimes b^{q},
\end{equation*}
these tensor products being over $\Z/\ell$. And the exact structure on $\filgalu(G_{\F},\Z/\ell)$ is defined as follows: a conflation is a short sequence $a\rightarrowtail b\twoheadrightarrow c$ with zero composition such that the associated graded sequences $0\to\gr^{n}a\to\gr^{n}b\to\gr^{n}c\to 0$ are all \emph{split} short exact in $\gal(G_{\F},\Z/\ell)$.

The étale realization $\reet_{\ell}$ sends $\Z/\ell(n)$ to $\mu_{\ell}(\overline{\F})^{\otimes n}\cong \Z/\ell$, using the primitive $\ell$th root of unity $\zeta$. Since the heart of $\D{G_{\F},\Z/\ell}^{b}_{c}$ with respect to the standard t-structure is closed under extensions, we see that the image of $\fil(\F,\Z/\ell)$ under the étale realization is contained in the heart. Applying the realization to the weight filtration \cref{weight-filtration} of an object $M\in\fil(\F,\Z/\ell)$ therefore yields a discrete Galois module together with a unipotent filtration. This describes the functor in the statement of \cref{etale-realization-filtered}. Note in particular that under this identification, the étale realization corresponds to forgetting the (unipotent) filtration of the discrete Galois module.
\begin{proof}[Proof of \cref{etale-realization-filtered}.] 
  The fact that this is an equivalence of exact categories is~\cite[3.1]{positselski:artin-tate-motives}. (This obviously relies crucially on the truth of the Beilinson-Lichtenbaum conjecture.) We only want to explain why it is compatible with the tensor structure.

Let $\T$ stand for the category $\dtm(\F,\Z/\ell)$, and correspondingly $\T^{\geq n}$ and $\T^{<n}$ for the subcategories considered above. We first note that the tensor product sends $\T^{\geq p}\times \T^{\geq q}$ into $\T^{\geq p+q}$, as follows immediately from $\Z/\ell(m)\otimes\Z/\ell(m')=\Z/\ell(m+m')$. Also, note that every object $M\in\T$ sits in a functorial triangle
\begin{equation}\label{weight-triangle}
  W^{\geq n}M\to M\to W^{<n}M\to W^{\geq n}M[1].
\end{equation}
Now fix integers $p,q$, and objects $M,N\in\T$. The two facts just mentioned imply that the canonical morphism $W^{\geq p}M\otimes W^{\geq q}N\to M\otimes N$ factors through $W^{\geq p+q}(M\otimes N)$, and this defines a natural transformation $W^{\geq p}\otimes W^{\geq q}\to W^{\geq p+q}\circ\otimes$. For $M$ and $N$ in $\fil(\F,\Z/\ell)$, there is then an induced morphism
\begin{equation}\label{reet-monoidal}
  \underline{\reet}_{\ell,\zeta}(M)\otimes\underline{\reet}_{\ell,\zeta}(N)\to \underline{\reet}_{\ell,\zeta}(M\otimes N)
\end{equation}
of filtered modules, using the fact that $\reet_{\ell}$ is a tensor functor. By functoriality of this construction, it is obvious that \cref{reet-monoidal} endows $\underline{\reet}_{\ell,\zeta}$ with a lax symmetric unital monoidal structure.

It remains to check that \cref{reet-monoidal} is in fact invertible for all $M,N\in\fil(\F,\Z/\ell)$. It is certainly invertible if $M=\Z/\ell(m)$ and $N=\Z/\ell(n)$. Moreover, since the tensor product in $\dtm(\F,\Z/\ell)$ is exact in both variables, and the étale realization an exact functor, it follows that, in both variables, the set of objects for which \cref{reet-monoidal} is invertible is closed under extensions. Thus the claim.
\end{proof}

Let us take stock: we have found an exact tensor category $\fil(\F,\Z/\ell)$ inside the tt-category $\dtm(\F,\Z/\ell)$ which we understand reasonably well, by \cref{etale-realization-filtered}. The question arises whether $\dtm(\F,\Z/\ell)$ is simply the (bounded) derived category of $\fil(\F,\Z/\ell)$. This kind of question is in general very hard to answer, not least because there is no canonical functor in either direction. Typically, \emph{this} problem can be solved if the triangulated category has a model. In our case we don't know whether $\dtm(\F,\Z/\ell)$ does, but it is a triangulated subcategory of $\dm(\F,\Z/\ell)$ which does. This is enough for Positselski to construct an exact functor
\begin{equation}
  \label{comparison-exact-triangulated}
  \pos:\D{\fil(\F,\Z/\ell)}^{b}\to\dtm(\F,\Z/\ell).
\end{equation}
Instead of invoking Koszulity of the graded Galois cohomology algebra $\mathrm{H}^{\bullet}(G_{\F},\Z/\ell)$ as in~\cite[9.1]{positselski:artin-tate-motives}, we will show directly that $\pos$ is an equivalence if $\F$ is algebraically closed (\cref{pos-exact-triangulated}).

Unfortunately, it is not clear whether \cref{comparison-exact-triangulated} is compatible with the tensor structures, and although it should be possible to construct a tensor equivalence using an extension of the methods employed in~\cite[Appendix~D]{positselski:artin-tate-motives}, we leave it as an open problem for now. Instead, we will establish the following weaker result which is enough for our purposes.

\begin{pro}\label{dtm-tt-equivalence}
  Assume that \cref{comparison-exact-triangulated} is an equivalence. Then it induces, together with the equivalence~$\underline{\reet}_{\ell,\zeta}$ of \cref{etale-realization-filtered}, a bijection
  \begin{equation*}
    \{\text{tt-ideals of }\dtm(\F,\Z/\ell)\}\longleftrightarrow\{\text{tt-ideals of }\D{\filgalu(G_{\F},\Z/\ell)}^{b}\}
  \end{equation*}
\end{pro}
\begin{proof}
  In order to construct the exact functor of~\cref{comparison-exact-triangulated}, a weaker version of filtered triangulated categories is used in~\cite{positselski:artin-tate-motives}. One can easily phrase the proof to be given here in this language but we prefer to work with derivators which we feel yields a conceptually more satisfying argument. The reason is that in the context of derivators, the bounded derived category of an exact category has the expected universal property~\cite[2.17]{porta:stable-derivators-universal}.

We place ourselves in the following abstract situation: $\iota:\A\hookrightarrow\T$ is an exact tensor subcategory of a (possibly large) tt-category $\T$ which is the base of a stable monoidal derivator $\DT$ (defined on finite categories). We also assume that $\hom_{\T}(X,Y[-1])=0$ for all $X,Y\in\A$.

The universal property mentioned above yields an exact morphism of derivators $\DF:\DA\to\DT$, where $\DA$ denotes the derivator with base $\D{\A}^{b}$. It induces an exact functor $\mathcal{F}:\D{\A}^{b}\to\T$ which is the identity on $\A$, and is unique up to unique isomorphism for these properties. Now fix an object $M\in\A$ and consider the two functors
\begin{align*}
  \mathcal{F}_{1}:\D{\A}^{b}&\to \T&\mathcal{F}_{2}:\D{\A}^{b}&\to \T\\
  N&\mapsto \mathcal{F}(M\otimes N)&N&\mapsto M\otimes \mathcal{F}(N)
\end{align*}
They coincide on $\A$, and satisfy the ``Toda conditions''~\cite[2.17]{porta:stable-derivators-universal}
\begin{equation*}
  \hom_{\T}(\mathcal{F}_{i}X,\mathcal{F}_{j}Y[n])=0,\quad i\leq j, n<0, X,Y\in\A,
\end{equation*}
because $M\in\A$ and $\A$ is closed under tensor products. It follows again from the universal property that the associated functors $\DF_{1},\DF_{2}:\DA\to\DT$ are canonically isomorphic. We deduce in particular that the following square commutes on the level of objects:
\begin{equation}\label{weak-monoidality}
  \vcenter{\xymatrix{\A\times\D{\A}^{b}\ar[r]^-{\iota\times\mathcal{F}}\ar[d]_{\otimes}&\T\times\T\ar[d]^{\otimes}\\
    \D{\A}^{b}\ar[r]_-{\mathcal{F}}&\T}}
\end{equation}

We apply this to $\A=\fil(\F,\Z/\ell)$, and $\DT$ the derivator associated to the stable monoidal model category modeling $\DM(\F,\Z/\ell)$. It follows from the uniqueness statement of~\cite[A.17]{positselski:artin-tate-motives} that $\DF:\DA\to\DT$ in this case induces an exact equivalence $\mathcal{F}:\D{\fil(\F,\Z/\ell)}^{b}\xrightarrow{\sim}\dtm(\F,\Z/\ell)\subset\DM(\F,\Z/\ell)$. In particular, it induces a bijection of thick subcategories. The proof will be complete once we check that
\begin{itemize}
\item tt-ideals in $\D{\filgalu(G_{\F},\Z/\ell)}^{b}$ are exactly the thick subcategories closed under tensoring with $\Z/\ell(n)$ (\ie $\Z/\ell$ placed in filtration degree $n$), $n\in\Z$.
\item tt-ideals in $\dtm(\F,\Z/\ell)$ are exactly the thick subcategories closed under tensoring with $\Z/\ell(n)$, $n\in\Z$.
\end{itemize}
Indeed, $\Z/\ell(n)\in\A$; now apply commutativity of the diagram in \cref{weak-monoidality}.

The two bullet points are an immediate consequence of the fact that in both cases the objects $\Z/\ell(n)$ generate the category as a thick subcategory~\cite[see][3.6]{gallauer:tt-fmod}.
\end{proof}

\begin{pro}
  \label{pos-exact-triangulated}
  Assume $\F$ is algebraically closed. Then the exact functor of~(\ref{comparison-exact-triangulated}) provides an equivalence
  \begin{equation*}
    \pos:\D{\Z/\ell}^{b}_{\scriptscriptstyle\mathrm{fil}}\xrightarrow{\sim}\dtm(\F,\Z/\ell).    
  \end{equation*}
  between the (bounded) filtered derived category of $\Z/\ell$-vector spaces and the triangulated category of Tate motives over $\F$ with coefficients in $\Z/\ell$.
\end{pro}
\begin{proof}
  Since $\F$ is algebraically closed, the exact category $\filgalu(G_\F,\Z/\ell)$ identifies with $\filgal(\Z/\ell)$, the category of filtered (finite-dimensional) vector spaces over $\Z/\ell$. Its (bounded) derived category is the classical filtered derived category $\D{\Z/\ell}^{b}_{\scriptscriptstyle\mathrm{fil}}$.
  The category $\dtm(\F,\Z/\ell)$ is generated, as a thick subcategory, by objects of the form $\Z/\ell(n)$. These lie in $\filgal(\Z/\ell)$, and $\D{\Z/\ell}^{b}_{\scriptscriptstyle\mathrm{fil}}$ is idempotent complete, therefore it suffices to prove fully faithfulness of the functor.

  Let $M,N$ be two complexes in $\filgal(\Z/\ell)$ and let us prove that
  \begin{equation*}
    \hom_{\D{\Z/\ell}^{b}_{\scriptscriptstyle\mathrm{fil}}}(M,N)\xrightarrow{\pos}\hom_{\dtm(\F,\Z/\ell)}(\pos(M),\pos(N))
  \end{equation*}
  is bijective. By induction on the length of these complexes and the five-lemma we reduce to $M,N$ shifts of objects in $\filgal(\Z/\ell)$. Similarly, by induction on the length of the filtration we reduce to $M=\Z/\ell$ and $N=\Z/\ell(n)[m]$, some $n,m\in\Z$. In other words, we need to show that
  \begin{equation*}
    \hom_{\D{\Z/\ell}^{b}_{\scriptscriptstyle\mathrm{fil}}}(\Z/\ell,\Z/\ell(n)[m])\xrightarrow{\pos}\hom_{\dtm(\F,\Z/\ell)}(\Z/\ell,\Z/\ell(n)[m])\cong\Hmm^{m,n}(\F,\Z/\ell)
  \end{equation*}
  is bijective. Both sides vanish whenever $m<0$. It is automatically bijective for $m=0$ since $\filgal(\Z/\ell)$ is a full subcategory of $\dtm(\F,\Z/\ell)$, and the same holds for $m=1$ since the subcategory is closed under extensions, \cf~\cite{dyer:exact-triangulated}. By the Beilinson-Lichtenbaum conjecture recalled above, the right-hand side vanishes when $m\geq 2$ (in fact, when $m\geq 1$) and the same is true for the left-hand side since the t-structure on $\D{\Z/\ell}^{b}_{\scriptscriptstyle\mathrm{fil}}$ is strongly hereditary, see~\cite[7.6]{gallauer:tt-fmod}. This completes the proof.
\end{proof}

\section{tt-primes}
\label{sec:primes}

In this section we are going to determine the prime ideals in the triangulated category of Tate motives over certain algebraically closed fields. This will use the results in the previous section, as well as the results in~\cite{gallauer:tt-fmod} where we determined the tt-geometry of filtered modules. As in the étale case (\cref{sec:etale}) we will first treat the case of finite coefficients; the case of rational coefficients is the same as in the étale case due to the equivalence $\dtm(\F,\Q)\simeq\dtmet(\F,\Q)$.
\begin{con}
  If not mentioned explicitly otherwise we assume that $\F$ is algebraically closed throughout this section. (The only exception is \cref{etale-sheafification-homeomorphism}.)
\end{con}

Let $\ell$ be a prime invertible in $\F$, and fix a primitive $\ell$th root of unity $\zeta\in\mu_{\ell}(\F)$ which we interpret as a morphism $\beta:\Z/\ell(0)\to\Z/\ell(1)$ as in \cref{sec:positselski}. From the results of the previous section and~\cite[7.8]{gallauer:tt-fmod} we deduce the following result.
\begin{pro}\label{dtm-ell-tt-ideals}
  The tt-category $\dtm(\F,\Z/\ell)$ has a unique non-trivial tt-ideal given by
  \begin{equation*}
    \ker(\reet_{\ell})=\langle\cone(\beta)\rangle.
  \end{equation*}
\end{pro}
\begin{proof}
We may apply \cref{dtm-tt-equivalence} and \cref{pos-exact-triangulated} to replace $\dtm(\F,\Z/\ell)$ by $\D{\Z/\ell}^{b}_{\scriptscriptstyle\mathrm{fil}}$. The étale realization functor then is identified with the functor $\pi:\D{\Z/\ell}^{b}_{\scriptscriptstyle\mathrm{fil}}\to\D{\Z/\ell}^{b}$ which forgets the filtration.

In~\cite[7.8]{gallauer:tt-fmod} we studied the tt-category $\D{\Z/\ell}^{b}_{\scriptscriptstyle\mathrm{fil}}$, and found that it has a unique non-trivial tt-ideal given by $\ker(\pi)=\langle\cone(\beta)\rangle$.\footnote{In the category of filtered $\Z/\ell$-vector spaces, $\Z/\ell(0)$ (respectively $\Z/\ell(1)$) is the 1-dimensional vector space placed in filtration degree 0 (respectively 1). The map $\beta:\Z/\ell(0)\to\Z/\ell(1)$ is then given by the identity on the underlying 1-dimensional vector space.}
\end{proof}

Fix the invertible object $\Z/\ell(1)$ in $\dtm(\F,\Z/\ell)$ and define the graded central ring
\begin{equation*}
  \R_{\ell}^{\bullet}=\hom_{\dtm(\F,\Z/\ell)}(\Z/\ell,\Z/\ell(\bullet)).
\end{equation*}

\begin{samepage}
  \begin{cor}\label{dtm-ell-tt-spectrum}\mbox{}
    \begin{enumerate}
    \item The graded central ring $\R_{\ell}^{\bullet}$ is canonically
      isomorphic to the polynomial ring~$\Z/\ell[\beta]$.
    \item The comparison morphism
      \begin{equation*}
        \rho^{\bullet}_{\F,\ell}:\spec(\dtm(\F,\Z/\ell))\to\spech(\Z/\ell[\beta])
      \end{equation*}
      is an isomorphism of locally ringed spaces.
    \end{enumerate}
  \end{cor}
\end{samepage}
\begin{proof}
  The first part can be deduced from the Beilinson-Lichtenbaum conjecture, recalled in \cref{sec:positselski}. For the second part, it suffices to show that the map is a homeomorphism. The map is a bijection by \cref{dtm-ell-tt-ideals}. And the only non-trivial open $\{\langle\cone(\beta)\rangle\}$ is mapped to the open subset $U(\beta)$.
\end{proof}

\begin{cor}\label{support-tate-ell}
  The support of $\Z/\ell(0)$ in $\spc(\dtm(\F,\Z))$ is the subspace
  \begin{equation*}
    \xymatrix{\ker(\bc)\\
      \ker(\reet_{\ell})\ar@{-}[u]}
  \end{equation*}
  where $\bc:\dtm(\F,\Z)\to\dtm(\F,\Z/\ell)$ is the change of coefficients functor, and the specialization relation is indicated by the line going upward.
\end{cor}
\begin{proof}
Since $\Z/\ell$ is perfect over $\Z$, the corresponding change of coefficients functor $\bc$ has a right adjoint~$\ff$. By~\cite[1.7]{balmer:surjectivity}, the image of $\spc(\bc)$ is precisely the support of $\ff\bc(\Z)=\Z/\ell(0)$.

On the other hand, using \cref{dtm-ell-tt-spectrum} we see that the image of $\spc(\bc)$ is the set of primes in the statement. It follows from \cref{cone-beta} below that the two primes are distinct since
\begin{equation*}
  \ff\cone(\beta)\in \ker(\reet_{\ell})\backslash\ker(\bc).
\end{equation*}
 The inclusion $\ker(\bc)\subset\ker(\reet_{\ell})$ gives the specialization relation in the statement. (And of course there can be no other specialization relation by continuity of $\spc(\bc)$.)
\end{proof}

We now have a good understanding of the tt-spectrum of $\dtm(\F,R)$ for coefficients $R=\Z/\ell$ and $R=\Q$, and our last step consists in patching these two cases together. For this we will use the results on étale Tate motives in the following form.
\begin{lem}\label{etale-sheafification-homeomorphism}
  Let $\F$ be a field of exponential characteristic $p$, and assume that for every $\ell\neq p$ prime, $\F$ contains a primitive $\ell$th root of unity (respectively, $4$th root of unity if $\ell=2$). Assume also that $\F$ satisfies \cref{vanishing-hypothesis}. Then the étale sheafification induces a map
  \begin{equation*}
    \spc(\Z[1/p])\xrightarrow{\spc(\aet)}\spc(\dtm(\F,\Z[1/p]),
  \end{equation*}
  which is a homeomorphism onto the subspace $\{\prrat,\pret{\ell}\mid \ell\neq p\}$ of torsion objects $\prrat$ and $\reet_{\ell}$-acyclics $\pret{\ell}$.
\end{lem}
\begin{proof}
  By \cref{spec-dtmet} and naturality of $\rho$ [\citealp[5.6]{balmer:sss}; recalled in \cref{sec:tt-sheaf}], the étale sheafification induces a section to $\rho$ on the level of spectra:
  \begin{equation*}
    \spc(\Z[1/p])\xrightarrow{\spc(\aet)}\spc(\dtm(\F,\Z[1/p])\xrightarrow{\rho}\spc(\Z[1/p]),
  \end{equation*}
  and therefore a homeomorphism onto its image. It is obvious that this image is precisely $\{\prrat,\pret{\ell}\mid\ell\neq p\}$.
\end{proof}

With this preparation we can now state and easily prove our main result in this section.
\newenvironment{absolutelynopagebreak}
  {\par\nobreak\vfil\penalty0\vfilneg
   \vtop\bgroup}
  {\par\xdef\tpd{\the\prevdepth}\egroup
   \prevdepth=\tpd}

 \begin{absolutelynopagebreak}
   \begin{thm}\label{dtm-primes}
     Let $\F$ be an algebraically closed field of exponential
     characteristic $p$ which satisfies
     \cref{vanishing-hypothesis}. The primes of $\dtm(\F,\Z[1/p])$ are
     depicted in the following diagram, including the specialization
     relations pointing upward.  \vspace{0.2cm}
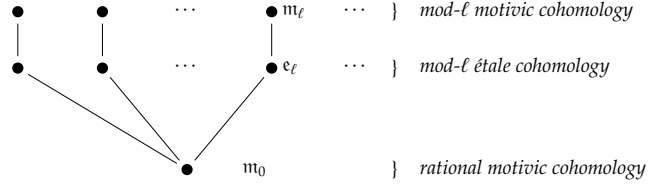
     \begin{center}
       \scalebox{0.75}{
    \begin{tikzpicture}[shorten >=3pt,shorten <=3pt, level
      1/.style={level distance=1.8cm}, level 2/.style={level
        distance=1cm}, every node/.style={inner sep=2pt},
      mod/.style={circle, fill=black,}, et/.style={circle, fill=black}, root/.style={circle, fill=black}, no/.style={edge from
        parent/.style={}} ]
      \node (root) [root,label={[label distance=8mm]right: $\prrat$}] {} [grow'=up] child {node [et] {} child
        {node [mod] {}} } child {node [et] {} child {node [mod] {}} }
      child[no] {node {$\cdots$} child {node {$\cdots$}} } child {node
        [et,label=right: {$\pret{\ell}$}] {} child {node [mod,label=right: {$\prmod{\ell}$}] {}} }child[no] {node {$\cdots$} child {node {$\cdots$}} } ;
      \begin{scope}[every node/.style={right}]
        \xdef\level{root} 
        \def\rightmostnode{root-3-1} 
        \foreach \text in {rational motivic cohomology, mod-{$\ell$} \'{e}tale
          cohomology, mod-{$\ell$} motivic cohomology} { \path (\level
           -|\rightmostnode) ++(35mm,0) node{$\}$} ++(5mm,0) node
           {\text}; \xdef\level{\level-1} }
       \end{scope}
       \coordinate (br) at (0,-4mm);
     \end{tikzpicture}
}
     \end{center}
     \vspace{0.2cm} Here, $\ell$ runs through all prime numbers
     different from $p$, and the prime tensor ideals are defined by
     the vanishing of the cohomology theories as indicated on the
     right.
   \end{thm}
 \end{absolutelynopagebreak}
\begin{proof}
  The central ring $\R=\hom_{\dtm(\F,\Z[1/p])}(\Z[1/p],\Z[1/p])$ is simply $\Z[1/p]$, and we get a canonical map
  \begin{equation*}
    \rho:\spc(\dtm(\F,\Z[1/p]))\to \spc(\Z[1/p])
  \end{equation*}
  which we analyse fiberwise, \ie we identify the primes in each fiber of $\rho$ with the corresponding primes in the statement of the Theorem.
  \begin{itemize}
  \item For a prime $\ell\neq p$ we have
    \begin{align*}
      \rho^{-1}(\langle\ell\rangle)&=\{\mathfrak{P}\mid\ell\in\rho(\mathfrak{P})\}\\
      &=\{\mathfrak{P}\mid\Z/\ell(0)=\cone(\ell)\notin\mathfrak{P}\}\\
      &=\supp(\Z/\ell(0)),
    \end{align*}
    and therefore we can apply \cref{support-tate-ell}.
  \item Over the generic point, the fiber is the spectrum of the central localization of $\dtm(\F,\Z[1/p])$ at $\Z\backslash 0$ (see~\cite[5.6]{balmer:sss} or \cref{spec-localization-cartesian-explicit}). Up to idempotent completion, this localization is $\dtm(\F,\Q)$, by \cref{Q-localization}. The contention now follows from \cref{spec-dtm-rational}.
  \end{itemize}
  At this point all that remains to be justified is the specialization relation $\prrat\rightsquigarrow\pret{\ell}$, for every $\ell\neq p$. This follows from \cref{etale-sheafification-homeomorphism}.
\end{proof}

\section{Classification of tt-ideals}
\label{sec:tt-ideals}
\begin{con}
  Throughout this section we fix a field $\F$ as in \cref{dtm-primes}.
\end{con}
We determined the prime ideals of $\dtm(\F,\Z[1/p])$, as well as the specialization relations among these. This is not enough to determine the tt-spectrum as a topological space, nor its Thomason subsets. The goal of this section is to remedy this, and then deduce the classification of the tt-ideals in $\dtm(\F,\Z[1/p])$.

The main input we need was already proved in \cref{uct}. Here is the tt-geometric content of this result.
\begin{pro}\label{cofinite-topology}
  The topology of $\spc(\dtm(\F,\Z[1/p]))$ is coarser than the cofinite topology.
\end{pro}
\begin{proof}
  Since the sets $\supp(M)$ with $M\in\dtm(\F,\Z[1/p])$ generate the closed subsets for the topology it suffices to show that if $\supp(M)$ is infinite then it is already the whole space. \cref{etale-sheafification-homeomorphism} implies that if $\supp(M)\cap \{\pret{\ell}\mid\ell\neq p\}$ is infinite then $\supp(M)$ is the whole space. Otherwise $\supp(M)\cap \{\prmod{\ell}\mid\ell\neq p\}$ must be infinite, \ie for infinitely many primes $\ell$, the mod-$\ell$ motivic cohomology of $M$ is non-trivial. By \cref{uct}, $M$ has nontrivial rational motivic cohomology, \ie $\prrat\in\supp(M)$, and this shows that $\supp(M)$ is the whole space.
\end{proof}

\begin{cor}\label{closed-subsets}
  For a proper subset $Z\subsetneq\spc(\dtm(\F,\Z[1/p]))$ the following are equivalent:
  \begin{enumerate}
  \item $Z$ is closed.
  \item $Z$ is finite and specialization closed.
  \end{enumerate}
\end{cor}

\begin{cor}\label{noetherian}
  The topological space $\spc(\dtm(\F,\Z[1/p]))$ is noetherian.
\end{cor}

We are now in a position to classify the tt-ideals in $\dtm(\F,\Z[1/p])$. In order to state the classification concisely, let us introduce the following notation:
\begin{itemize}
\item $\PP=\{\text{prime numbers }\ell\text{ different from }p\}$,
\item for every $\ell\in\PP$, choose a Bott element $\beta_{\ell}:\Z/\ell\to \Z/\ell(1)$ in $\dtm(\F,\Z/\ell)$, \ie a primitive $\ell$th root of unity; we denote abusively by $\cone(\beta_{\ell})$ the image of its cone in $\dtm(\F,\Z[1/p])$ under the right adjoint $\ff$ of the change of coefficients functor.
\end{itemize}

\begin{thm}\label{tt-ideals}
   Let $\F$ be an algebraically closed field of exponential characteristic $p$ which satisfies \cref{vanishing-hypothesis}. The following two maps are inverses to each other and set up a bijection
  \begin{align*}
    \{\text{proper tt-ideals in }\dtm(\F,\Z[1/p])\}&\longleftrightarrow\{\text{subsets } \mathcal{E}\subset\mathcal{M}\subset\PP\}\\
    \mathcal{I}&\longmapsto \{\ell\mid\Hmet^{\bullet}(\mathcal{I},\Z/\ell)\neq 0\}\subset\{\ell\mid\Hmm^{\bullet,\bullet}(\mathcal{I},\Z/\ell)\neq 0\}\\
    \langle\cone(\beta_{\ell}),\Z/\ell'(0)\mid\ell\in\mathcal{M},\ell'\in\mathcal{E}\rangle&\longmapsfrom (\mathcal{E}\subset\mathcal{M})
  \end{align*}
\end{thm}
\begin{proof}
  Since $\spc(\dtm(\F,\Z[1/p]))$ is noetherian (\cref{noetherian}), the Thomason subsets are precisely the specialization closed ones. It is then clear that the proper Thomason subsets correspond bijectively to $\{\text{subsets } \mathcal{E}\subset\mathcal{M}\subset\PP\}$. One now applies~\cite[4.10]{balmer:spectrum}.
\end{proof}

\begin{exa}\label{etale-sheafification-kernel}
  The étale sheafification functor
  \begin{equation*}
    \aet:\dtm(\F,\Z[1/p])\to\dtmet(\F,\Z[1/p])
  \end{equation*}
  is a non-trivial tt-functor and its kernel therefore a proper tt-ideal. It corresponds to the subsets $\emptyset=\mathcal{E}\subset\mathcal{M}=\PP$. We must then have
  \begin{equation*}
    \ker(\aet)=\langle\cone(\beta_{\ell})\mid \ell\in\PP\rangle.
  \end{equation*}
In fact, we will prove in \cref{sec:bott} that $\aet$ is a Verdier localization  of the tt-category $\dtm(\F,\Z[1/p])$ at $\cone(\beta_{\ell})$ for all $\ell\neq p$.
\end{exa}

\section{Structure sheaf}
\label{sec:spec}
\begin{con}
  We continue to denote by $\F$ an algebraically closed field of exponential characteristic $p$, satisfying \cref{vanishing-hypothesis}.
\end{con}
At this point we know $\spec(\dtm(\F,\Z[1/p]))$ as a topological space, and in this last section we want to describe the structure sheaf of this locally ringed space. We denote it simply by $\mathcal{O}_{\F}$.

\begin{pro}\label{spc-sheaf}
  Let $\mathcal{O}_{\F}^{\et}$ denote the structure sheaf on $\dtmet(\F,\Z[1/p])$ (which is according to \cref{spec-dtmet} essentially just $\Z[1/p]$). The canonical map
  \begin{equation*}
    \mathcal{O}_{\F}\to\spc(\aet)_{*}\mathcal{O}_{\F}^{\et}
  \end{equation*}
  is an isomorphism.
\end{pro}
\begin{proof}
By \cref{bott-invert}, we know that the functor $\aet$ is a Verdier localization, and by \cref{etale-sheafification-kernel}, the kernel is $\langle\cone(\beta_{\ell})\mid\ell\neq p\rangle$. By \cref{spec-localization-sheaf}, the map in the statement of the Proposition is an isomorphism on all stalks in the image of $\spc(\aet)$, \ie on all non-closed points. For a closed point $\prmod{\ell}$, $\ell\neq p$, we may localize at $\ell$ and consider the functor
\begin{equation*}
  \aet:\dtm(\F,\Z_{\langle\ell\rangle})\to\dtmet(\F,\Z_{\langle\ell\rangle})
\end{equation*}
instead (cf.~\cref{spec-localization-cartesian-explicit}). In that case the prime $\prmod{\ell}$ is the zero ideal (the category $\dtm(\F,\Z_{\langle\ell\rangle})$ is local; cf.\ \cref{cyclotomic-field}) hence
    \begin{equation*}
      \mathcal{O}_{\F,\prmod{\ell}}=\End_{\dtm(\F,\Z_{\langle\ell\rangle})}(\Z_{\langle\ell\rangle})=\Z_{\langle\ell\rangle}.
    \end{equation*}
    Again since the space $\spc(\dtm(\F,\Z_{\langle\ell\rangle}))$ is local, this is also the stalk of $\spc(\aet)_{*}\mathcal{O}_{\F}^{\et}$ at $\prmod{\ell}$. The morphism induced between these stalks is clearly an isomorphism which concludes the proof.
\end{proof}

\begin{cor}The stalks at the primes are
  \begin{align*}
    \mathcal{O}_{\F,\pret{\ell}}=\mathcal{O}_{\F,\prmod{\ell}}=\Z_{\langle\ell\rangle},&&\mathcal{O}_{\F,\prmod{0}}=\Q.
  \end{align*}
\end{cor}

\begin{cor}
  The locally ringed space $\spec(\dtm(\F,\Z[1/p]))$ is not a scheme.
\end{cor}
\begin{proof}
  The canonical functor from schemes to locally ringed spaces preserves fiber products. Localizing at $S=\Z\backslash\langle\ell\rangle$, it would follow from \cref{spec-localization-cartesian}, that $\spec(\dtm(\F,\Z_{\langle\ell\rangle}))$ is a scheme as well. Since it is local it would have to be affine, the spectrum of $\Z_{\langle\ell\rangle}$. But the latter has two, not three, points.
\end{proof}

\appendix{}
\section{Some remarks on Balmer's structure sheaf}
\label{sec:tt-sheaf}
Balmer in~\cite{balmer:spectrum} (see also~\cite{balmer:sss}) associates to every (small) rigid tt-category $\T$ a locally ringed space $\spec(\T)$, its tt-spectrum. It is fairly obvious that this actually extends canonically to a functor satisfying certain good properties. Since we have not seen this explained in the literature and since we need it in the body of the text, let us spell out the details here. For this section only, we assume that all tt-categories are rigid.
\begin{lem}\label{spec-functorial}
  Balmer's construction canonically extends to a contravariant functor $\spec:\rigttCat\to\Lrs$ from the category of rigid tt-categories to the category of locally ringed spaces.
\end{lem}
\begin{proof}
  That $\spc$ is a contravariant functor $\rigttCat\to\Tpl$ is proved in~\cite[3.6]{balmer:spectrum} so that we need only consider the structure sheaves. Recall that these are defined as the sheafification of a presheaf $_{p}{\mathcal O}$ on the canonical base for the topology, whose definition we now recall. Let $\T$ be a tt-category, and $a\in \T$. Sections of $_{p}{\mathcal O_{\T}}$ over $U(a)$ are given by endomorphisms of the unit in $\T/\langle a\rangle$. Given an inclusion $U(b)\subset U(a)$, we have $\langle a\rangle\subset\langle b\rangle$ from which a functor $\T/\langle a\rangle\to\T/\langle b\rangle$ and then an induced morphism of rings.

We note that the association $U(a)\mapsto \T/\langle a\rangle$ can be made into a functor $\T/-$ with values in $\rigttCat$. We can then compose with the functor $\R_{-}=\End_{-}(\one):\rigttCat\to \Rng$ to the category of rings, and this defines the presheaf $_{p}{\mathcal O}_{\T}$ on the distinguished base for the topology.

Given a tt-functor $F:\T\to\T'$, denote its induced continuous map $\spc(F)$ by~$f$. We have $f^{-1}(U(a))=U(Fa)$ and $F:\T/\langle a\rangle\to\T'/\langle Fa\rangle$ thus a natural transformation $F/-:\T/-\to\T'/F(-)$. Whiskering with $\End_{-}(\one)$ we obtain a morphism of presheaves of rings
  \begin{equation*}
    _{p}{\mathcal O}_{\T}\to f_{*}{}_{p}{\mathcal O}_{\T'}.
  \end{equation*}
After sheafifying we clearly obtain a functor $\rigttCatop\to \Rs$, the category of ringed spaces. The objects are sent to locally ringed spaces, by~\cite[6.6]{balmer:sss}, and it remains to check that the morphisms are local. This can be checked on the level of presheaves. Fix a prime $\mathfrak{P}\in\spec(\T')$ and let $\mathfrak{Q}=f(\mathfrak{P})=F^{-1}(\mathfrak{P})$. Using~\cite[6.5]{balmer:sss} we see that the morphism on stalks at these two points naturally identifies with the morphism $\R_{\T/\mathfrak{Q}}\to\R_{\T'/\mathfrak{P}}$ induced by $F:\T/\mathfrak{Q}\to\T'/\mathfrak{P}$. But this functor is conservative by definition, \ie detects isomorphisms, in particular automorphisms of $\one$ hence $\R_{\T/\mathfrak{Q}}\to\R_{\T'/\mathfrak{P}}$ is local.
\end{proof}

\begin{lem}\label{comparison-natural}
  The comparison morphism $\rho:\spec(\T)\to\spec(\R_{\T})$ defines a natural transformation of functors $\rigttCatop\to\Lrs$.
\end{lem}
\begin{proof}
  Naturality on the level of topological spaces is~\cite[5.3]{balmer:sss}. Also, $\rho$ is a morphism of locally ringed spaces, by~\cite[6.11]{balmer:sss}. It remains to check naturality on the level of sheaves, or indeed, presheaves. In other words, for $F:\T\to\T'$ we need to show commutativity of the square
  \begin{equation*}
    \xymatrix{_{p}{\cal O}_{\T'}(U(\cone(Fr)))&_{p}{\cal O}_{\T}(U(\cone(r)))\ar[l]_-{F}\\
_{p}{\cal O}_{\R_{\T'}}(D(Fr))\ar[u]&_{p}{\cal O}_{\R_{\T}}(D(r))\ar[u]\ar[l]^-{F}}
  \end{equation*}
  where by definition~\cite[6.10]{balmer:sss} the vertical arrows are isomorphisms, identifying both rings with $\R_{\T'}[1/Fr]$, respectively $\R_{\T}[1/r]$. By the universal property of localization at the level of rings, it suffices to prove that the diagram commutes on the image of $\R_{\T}\to {}_{p}{\cal O}_{\R_{\T}}(D(r))$. But for $s\in\R_{\T}=\End_{\T}(\one)$, the image under both possible paths traversing the square is simply $Fs$ considered as an endomorphism of $\one\in \T'/\langle\cone(Fr)\rangle$.
\end{proof}

\begin{lem}\label{spec-localization-cartesian}
  Let $\T$ be a tt-category, and $S\subset \R_{\T}$ a multiplicative system. Then the following square is cartesian in $\Lrs$:
  \begin{equation*}
    \xymatrix{\spec(S^{-1}\T)\ar@{^{(}->}[r]^{\spc(Q)}\ar[d]_{\rho_{S^{-1}\T}}&\spec(\T)\ar[d]^{\rho_{\T}}\\
      \spec(\R_{S^{-1}\T})=\spec(S^{-1}\R_{\T})\ar@{^{(}->}[r]&\spec(\R_{\T})}
  \end{equation*}
(Here $Q$ denotes the canonical localization functor $Q:\T\to S^{-1}\T$.)
\end{lem}
Before giving the proof let us recall that for this type of diagram (where the bottom horizontal map is an isomorphism on stalks) the fiber product in the category $\Lrs$ is simple to describe: it coincides with the fiber product in the category $\Rs$~\cite[Cor. 11]{gillam:localization-ringed-spaces}. So, this result can be made more explicit as follows.
\begin{cor}\label{spec-localization-cartesian-explicit}
  In the situation of \cref{spec-localization-cartesian}, $\spec(S^{-1}\T)$ maps homeomorphically onto $\{\mathfrak{P}\in\spec(\T)\mid\rho_{\T}(\mathfrak{P})\cap S=\emptyset\}$, and its structure sheaf identifies with the restriction of ${\cal O}_{\T}$ to this subset.
\end{cor}

\begin{proof}
  By the remarks just made, \cref{spec-localization-cartesian} and \cref{spec-localization-cartesian-explicit} are equivalent.  Moreover, \cite[5.6]{balmer:sss} shows that the square is cartesian on the level of sets. Both $S^{-1}\T$ and $\T$ have the ``same'' base for the topology, namely $U(a)$ where $a\in\T$, and we see that the diagram is cartesian on the level of topological spaces as well.

The square is commutative in $\Lrs$, by \cref{comparison-natural}. Consequently we obtain a canonical morphism of locally ringed spaces $\spec(S^{-1}\T)\to X$ where $X$ is the fiber product in $\Lrs$. To show that it is an isomorphism we check that it is so on stalks. If $\mathfrak{P}\in\spec(\T)$ such that $\rho_{\T}(\mathfrak{P})\cap S=\emptyset$, and $S^{-1}\mathfrak{P}$ is the corresponding prime in $S^{-1}\T$, then the stalk of ${\cal O}_{S^{-1}\T}$ at $S^{-1}\mathfrak{P}$ is
\begin{equation*}
  {\cal O}_{S^{-1}\T,S^{-1}\mathfrak{P}}=\R_{S^{-1}\T/S^{-1}\mathfrak{P}}=\R_{\T/\mathfrak{P}}={\cal O}_{\T,\mathfrak{P}},
\end{equation*}
and we conclude.
\end{proof}

\begin{rmk}\label{spec-localization-sheaf}
  The last argument in this proof also shows that for any Verdier localization $\mathcal{T}\to\mathcal{T}/\mathcal{K}$ the induced map on spectra (which is a homeomorphism onto its image)
  \begin{equation*}
    \spc(\mathcal{T}/\mathcal{K})\to\spc(\mathcal{T})
  \end{equation*}
  identifies the structure sheaf on the domain with the restriction of the structure sheaf on the codomain.
\end{rmk}
\section{The motivic Bott element and change of coefficients}
\label{sec:bott-coefficients}

In this section we will perform some computations regarding how the Bott elements behave under certain changes of coefficients. Our main goal is \cref{bott-equality} which states that inverting a Bott element of any prime power order is equally good. We fix a field $\F$, a localization $\Z\subset R\subset\Q$, and two integers $1 < n,N$ such that $n\mid N$ and $N$ is invertible in $\F$. We also assume that $\F$ contains a primitive $N$th root of unity $\zeta_{N}$, and we let $\zeta_{n}=\zeta_{N}^{N/n}$, a primitive $n$th root of unity. Finally, we use a subscript $(-)_{R}$ to denote tensoring with $R$.

For any positive integer $k$, the triangle
\begin{equation}
R\xrightarrow{k}R\to R/k\to R[1]\label{eq:bockstein}
\tag{Bk}\end{equation}
gives rise to a long exact sequence
\begin{multline*}
  \to\hom_{\dm(\F,R)}(R,R(1))\to\hom_{\dm(\F,R)}(R,R/k(1))\to\\\hom_{\dm(\F,R)}(R,R(1)[1])\xrightarrow{k}\hom_{\dm(\F,R)}(R,R(1)[1])\to
\end{multline*}
which identifies with
\begin{multline*}
  0\to\hom_{\dm(\F,R)}(R,R/k(1))\to\F^{\times}_{R}\xrightarrow{k}\F^{\times}_{R}\to,
\end{multline*}
and we see that $\hom_{\dm(\F,R)}(R,R/k(1))=\mu_{k}(\F)_{R}$. By adjunction, also
\begin{equation*}
  \hom_{\dm(\F,R/k)}(R/k,R/k(1))=\mu_{k}(\F)_{R}.
\end{equation*}
We can therefore interpret the $N$th root of unity $\zeta_{N}$ as a morphism $\beta_{N}:R/N\to R/N(1)$ in $\dm(\F,R/N)$. This is called the motivic Bott element (with $R/N$-coefficients). Similarly we obtain $\beta_{n}:R/n\to R/n(1)$ in $\dm(\F,R/n)$.

Fix the following notation
\begin{equation*}
  \xymatrix{R\ar[r]_-{\Gamma}\ar@/^1pc/[rr]^{\gamma}&R/N\ar[r]_{\pi}&R/n}
\end{equation*}
with associated change of coefficients adjunctions $\Gamma^{*}\dashv\Gamma_{*}$, $\gamma^{*}\dashv\gamma_{*}$, $\pi^{*}\dashv\pi_{*}$.
\begin{lem}\label{compatibility-bott}
  The following square in $\dm(\F,R)$ commutes:
  \begin{equation*}
    \xymatrix{R/N\ar[r]^{\pi}\ar[d]_{\Gamma_{*}\beta_{N}}&R/n\ar[d]^{\gamma_{*}\beta_{n}}\\
      R/N(1)\ar[r]_{\pi}&R/n(1)}
  \end{equation*}
\end{lem}
\begin{proof}
  Applying $\hom_{\dm(\F,R)}(-,R/n(1))$ to \hyperref[eq:bockstein]{(BN)} we obtain part of a long exact sequence
  \begin{equation*}
    0\to \hom_{\dm(\F,R)}(R/N,R/n(1))\to\hom_{\dm(\F,R)}(R,R/n(1))=\mu_{n}(\F)_{R}\xrightarrow{N=0}\mu_{n}(\F)_{R}.
  \end{equation*}
  In particular, it suffices to show commutativity of the square after precomposing with $\Gamma:R\to R/N$. It also shows that $\gamma_{*}\beta_{n}\circ\pi$ corresponds to $\zeta_{n}\in\mu_{n}(\F)_{R}$.

  Now, $\pi$ fits into a morphism of triangles
  \begin{equation*}
    \xymatrix{R\ar[r]^{N}\ar[d]_{N/n}&R\ar[r]\ar@{=}[d]&R/N\ar[r]\ar[d]_{\pi}&R[1]\ar[d]_{N/n}\\
R\ar[r]_{n}&R\ar[r]&R/n\ar[r]&R[1]}
  \end{equation*}
  and applying $\hom_{\dm(\F,R)}(R,-(1))$ we obtain
  \begin{equation*}
    \xymatrix{\mu_{N}(\F)_{R}\ar[d]_{\pi}\ar[r]&\F^{\times}_{R}\ar[d]^{N/n}\\
      \mu_{n}(\F)_{R}\ar[r]&\F^{\times}_{R}}
  \end{equation*}
  thus $\pi\circ\Gamma_{*}\beta_{N}$ corresponds to $\zeta_{N}^{N/n}=\zeta_{n}\in\mu_{n}(\F)_{R}$. This concludes the proof.
\end{proof}

\begin{lem}\label{compatibility-bott-2}
  Let $m=N/n$. The following square in $\dm(\F,R)$ commutes:
  \begin{equation*}
    \xymatrix{R/n\ar[r]^{m}\ar[d]_{\gamma_{*}\beta_{n}}&R/N\ar[d]^{\Gamma_{*}\beta_{N}}\\
      R/n(1)\ar[r]_{m}&R/N(1)}
  \end{equation*}
\end{lem}
\begin{proof}
  We have a morphism of triangles
  \begin{equation*}
    \xymatrix{R\ar@{=}[d]\ar[r]^{n}&R\ar[d]_{m}\ar[r]&R/n\ar[d]^{m}\ar[r]&R[1]\ar@{=}[d]\\
      R\ar[r]_{N}&R\ar[r]&R/N\ar[r]&R[1]}
  \end{equation*}
  which implies the commutativity of the bottom half of the following diagram:
  \begin{equation*}
    \xymatrix{\hom(R/n,R/n(1))\ar[r]^{m}\ar[d]&\hom(R/n,R/N(1))\ar[d]\\
      \hom(R,R/n(1))\ar[r]^{m}\ar[d]&\hom(R,R/N(1))\ar[d]\\
      \hom(R,R(1)[1])\ar@{=}[r]&\hom(R,R(1)[1])}
  \end{equation*}
  The upper half clearly commutes and the vertical arrows are injections. We deduce that $m\circ\gamma_{*}\beta_{n}$ corresponds to $\zeta_{n}\in \F^{\times}_{R}$.

  Next consider the following diagram:
  \begin{equation*}
    \xymatrix{\hom(R/N,R/N(1))\ar[r]^{m}\ar[d]&\hom(R/n,R/N(1))\ar[d]\\
      \hom(R,R/N(1))\ar[r]^{m}\ar[d]&\hom(R,R/N(1))\ar[d]\\
      \hom(R,R(1)[1])\ar[r]^{m}&\hom(R,R(1)[1])}
  \end{equation*}
  The commutativity of the upper half again follows from the morphism of triangles above, while the lower half clearly commutes. The vertical arrows are injections and we deduce that $\Gamma_{*}\beta_{N}\circ m$ corresponds to $\zeta_{N}^{m}=\zeta_{n}$ thus the claim.
\end{proof}

\begin{lem}\label{adjunction-faithful}
  Let $F\dashv G:\mathcal{C}\to\mathcal{D}$ be an adjunction. For any $c\in\mathcal{C}$ and $d\in\mathcal{D}$ the following map is injective:
  \begin{equation*}
    \hom_{\mathcal{D}}(Fc,d)\xrightarrow{G}\hom_{\mathcal{C}}(GFc,Gd).
  \end{equation*}
\end{lem}
\begin{proof}
  By adjunction, the target is identified with $\hom_{\mathcal{D}}(FGFc,d)$ and under this identification, the map is induced by the counit $FGFc\to Fc$ which is a split epimorphism (the splitting is given by the unit of the adjunction). Thus the claim.
\end{proof}

\begin{lem}
  \label{bfb} For any $M\in\dm(\F,R)$ we have:
  \begin{equation*}
    \gamma^{*}\gamma_{*}\gamma^{*}M\cong \gamma^{*}M\oplus\gamma^{*}M[1]
  \end{equation*}
\end{lem}
\begin{proof}
  Tensoring $M$ with \hyperref[eq:bockstein]{(Bn)} and applying $\gamma^{*}$ we obtain a triangle
  \begin{equation*}
    \gamma^{*}M\xrightarrow{n}\gamma^{*}M\xrightarrow{\pi}\gamma^{*}\gamma_{*}\gamma^{*}M\to\gamma^{*}M[1],
  \end{equation*}
  and since the first map is zero (the category $\dm(\F,R/n)$ is $\Z/n$-linear), the triangle splits and the Lemma follows.
\end{proof}

\begin{lem}
  \label{fbf} For any $M\in\dm(\F,R/n)$ we have:
  \begin{equation*}
    \gamma_{*}\gamma^{*}\gamma_{*}M\cong \gamma_{*}M\oplus\gamma_{*}M[1]
  \end{equation*}
\end{lem}
\begin{proof}
  Tensoring $\gamma_{*}M$ with \hyperref[eq:bockstein]{(Bn)} we obtain a triangle
  \begin{equation*}
    \gamma_{*}M\xrightarrow{n}\gamma_{*}M\xrightarrow{\pi}\gamma_{*}\gamma^{*}\gamma_{*}M\to\gamma_{*}M[1],
  \end{equation*}
  and since the first map is zero the triangle splits and the Lemma follows.
\end{proof}

\begin{lem}\label{cone-beta} We have in $\dm(\F,R/n)$:
  \begin{equation*}
    \gamma^{*}\gamma_{*}\cone(\beta_{n})=\cone(\beta_{n})\oplus\cone(\beta_{n})[1]
  \end{equation*}
\end{lem}
\begin{proof}
  We have
  \begin{equation*}
    \gamma^{*}\gamma_{*}\cone(\beta_{n})=\cone\left(\gamma^{*}\gamma_{*}\gamma^{*}R\xrightarrow{\gamma^{*}\gamma_{*}\beta_{n}}\gamma^{*}\gamma_{*}\gamma^{*}R(1)\right),
  \end{equation*}
  and by \cref{bfb}, $\gamma^{*}\gamma_{*}\beta_{n}$ is a morphism between $\gamma^{*}R\oplus\gamma^{*}R[1]$ and $\gamma^{*}R(1)\oplus\gamma^{*}R(1)[1]$. Such a morphism is therefore described by a $2\times 2$-matrix, whose diagonal entries ``are'' elements of $\mu_{n}(\F)_{R}$, while the off-diagonal entries necessarily vanish. To describe the non-trivial entries we can do so after applying $\gamma_{*}$, by \cref{adjunction-faithful}. But by \cref{fbf}, we have $\gamma_{*}\gamma^{*}\gamma_{*}\cone(\beta_{n})=\gamma_{*}\cone(\beta_{n})\oplus\gamma_{*}\cone(\beta_{n})[1]$ which completes the proof.
\end{proof}

\begin{lem}
  \label{nN} We have in $\dm(\F,R/n)$:
  \begin{equation*}
\gamma^{*}\Gamma_{*}\cone(\beta_{N})\in\langle\cone(\beta_{n})\rangle
\end{equation*}
\end{lem}
\begin{proof}
  First, we have
  \begin{align*}
    \gamma^{*}\Gamma_{*}\cone(\beta_{N})&=\pi^{*}\Gamma^{*}\Gamma_{*}\cone(\beta_{N})\\
    &=\pi^{*}\cone(\beta_{N})\oplus\pi^{*}\cone(\beta_{N})[1]
  \end{align*}
  by \cref{cone-beta}. It now suffices to show that
  \begin{equation*}
    \pi^{*}:\hom_{\dm(\F,R/N)}(R/N,R/N(1))\to\hom_{\dm(\F,R/n)}(R/n,R/n(1))
  \end{equation*}
  maps $\beta_{N}$ to $\beta_{n}$. This follows easily from \cref{compatibility-bott}.
\end{proof}

\begin{lem}
  \label{Nn} Assume the existence of a primitive $(N\cdot n)$th root of unity in $\F$. We then have in $\dm(\F,R/N)$:
  \begin{equation*}
\Gamma^{*}\gamma_{*}\cone(\beta_{n})\in\langle\cone(\beta_{N})\rangle
\end{equation*}
\end{lem}
\begin{proof}
  More precisely we are going to prove that the cone $C$ of multiplication by $n$ on $\cone(\beta_{N})$ is $\cone(\beta_{n})\oplus\cone(\beta_{n})[1]$.

  Consider the following commutative diagram of solid arrows:
  \begin{equation}\label{nN-diagram}
    \xymatrix{\Gamma^{*}R\ar[r]^{n}\ar[d]_{\beta_{N}}&\Gamma^{*}R\ar[d]_{\beta_{N}}\ar[r]&\Gamma^{*}\gamma_{*}\gamma^{*}R\ar[r]\ar@{.>}[d]&\Gamma^{*}R[1]\ar[d]_{\beta_{N}[1]}\\
      \Gamma^{*}R(1)\ar[r]_{n}&\Gamma^{*}R(1)\ar[r]&\Gamma^{*}\gamma_{*}\gamma^{*}R(1)\ar[r]&\Gamma^{*}R(1)[1]}
  \end{equation}
  We want to prove that $\Gamma^{*}\gamma_{*}\beta_{n}$ makes this diagram commutative. Taking cones of the vertical maps and applying the octahedral axiom the Lemma would then be proved. We may prove commutativity after applying $\Gamma_{*}$ by \cref{adjunction-faithful}. In fact, we will prove that after applying $\Gamma_{*}$, the morphism $\Gamma^{*}\gamma_{*}\beta_{n}$ is identified with $\alpha$, the ``natural'' cone coming from a model. For this we will work in the homotopy category of bounded complexes of Nisnevich sheaves with transfers (\ie before $\mathbb{A}^{1}$-localization). As a model for $R(1)$ we will use $\mathcal{O}^{\times}_{R}[-1]$. The Bott element $\beta_{N}$ can then be modeled by the following morphism of complexes, where the last term is in degree $-1$:
\begin{equation*}
  \xymatrix{R\ar[d]_{N}\\
    R\ar[r]^{\zeta_{N}}&\mathcal{O}^{\times}_{R}\ar[d]^{N}\\
    &\mathcal{O}^{\times}_{R}}
\end{equation*}
Taking the mapping cone of multiplication by $n$ we obtain the following model for $\alpha$:
\begin{align*}
  \xymatrix@C=70pt{R\ar[d]_{
      \begin{pmatrix}
        -N\\
        -n
      \end{pmatrix}
}\\
R\oplus R\ar[d]_{
  \begin{pmatrix}
    -n&N
  \end{pmatrix}
}\ar[r]^{
        \begin{pmatrix}
          \zeta_{N}&0
        \end{pmatrix}
}&\mathcal{O}^{\times}_{R}\ar[d]^{
      \begin{pmatrix}
        -N\\
        -n
      \end{pmatrix}
}\\
R\ar[r]_{
  \begin{pmatrix}
    0\\\zeta_{N}
  \end{pmatrix}
}&\mathcal{O}^{\times}_{R}\oplus \mathcal{O}^{\times}_{R}\ar[d]^{
  \begin{pmatrix}
    -n&N
  \end{pmatrix}
}\\
&\mathcal{O}^{\times}_{R}}
\end{align*}
the last term being in degree $-1$.

Now, since $n\mid N$, the domain and the codomain split into direct sums of two-term complexes, and under these identifications, $\alpha$ is identified with
\begin{equation*}
  \xymatrix@C=70pt{R\ar[d]_{
      \begin{pmatrix}
        0\\-n
      \end{pmatrix}
}\\
R\oplus R\ar[d]_{
  \begin{pmatrix}
    n&0
  \end{pmatrix}
}\ar[r]^{
  \begin{pmatrix}
    \zeta_{N}^{-1}&\zeta_{N}^{N/n}
  \end{pmatrix}
}&\mathcal{O}^{\times}_{R}\ar[d]^{
  \begin{pmatrix}
    0\\-n
  \end{pmatrix}
}\\
R\ar[r]_{
  \begin{pmatrix}
    \zeta_{N}^{N/n}\\\zeta_{N}
  \end{pmatrix}
}&\mathcal{O}^{*}_{R}\oplus\mathcal{O}^{\times}_{R}\ar[d]^{
  \begin{pmatrix}
    n&0
  \end{pmatrix}
}\\
&\mathcal{O}^{\times}_{R}}
\end{equation*}
Let $\zeta_{Nn}$ be an $n$th root of $\zeta_{N}$, and define the homotopy $\zeta_{Nn}^{-1}:R\to\mathcal{O}^{\times}_{R}$. It shows that $\alpha$ is homotopic to the map $
\begin{pmatrix}
  0&\zeta_{n}
\end{pmatrix}$ (in degree 1), $
\begin{pmatrix}
  \zeta_{n}\\0
\end{pmatrix}$ (in degree 0), \ie a model for $\Gamma_{*}\Gamma^{*}\gamma_{*}\beta_{n}$. This concludes the proof.
\end{proof}

For the last two results we specialize to the case $N=\ell^{m}$ and $n=\ell$ for some $m\geq 1$.
\begin{lem}\label{bott-nN-integrally}
We have in $\dm(\F,R)$:
  \begin{equation*}
    \Gamma_{*}\cone(\beta_{\ell^{m}})\in\langle\gamma_{*}\cone(\beta_{\ell})\rangle
  \end{equation*}
\end{lem}
\begin{proof}
  The proof is by induction on $m$. Assume $m>1$ and consider the following diagram in $\dtm(\F,R)$:
  \begin{equation}\label{beta-induction}
    \xymatrix{R/\ell^{m-1}\ar[r]^{\ell}\ar[d]_{\beta_{\ell^{m-1}}}&R/\ell^{m}\ar[d]_{\beta_{\ell^{m}}}\ar[r]&R/\ell\ar[d]_{\beta_{\ell}}\ar[r]^{\delta}&R/\ell^{m-1}[1]\ar[d]_{\beta_{\ell^{m-1}}[1]}\\
R/\ell^{m-1}(1)\ar[r]^{\ell}&R/\ell^{m}(1)\ar[r]&R/\ell(1)\ar[r]^{\delta}&R/\ell^{m-1}(1)[1]}
  \end{equation}
The two rows are triangles thus if the diagram commutes we can take cones of the vertical maps and the induction hypothesis will allow to conclude.

We now proceed to describe this diagram using the same model as in the previous proof. Commutativity of the first square is \cref{compatibility-bott-2}. We may therefore compute the induced morphism on the mapping cones of multiplication by $\ell$:
\begin{align*}
  \xymatrix@C=70pt{R\ar[d]_{
      \begin{pmatrix}
        -\ell^{m-1}\\
        -1
      \end{pmatrix}
}\\
R\oplus R\ar[d]_{
  \begin{pmatrix}
    -\ell&\ell^{m}
  \end{pmatrix}
}\ar[r]^{
        \begin{pmatrix}
          \zeta_{\ell^{m-1}}&0
        \end{pmatrix}
}&\mathcal{O}^{\times}_{R}\ar[d]^{
      \begin{pmatrix}
        -\ell^{m-1}\\
        -1
      \end{pmatrix}
}\\
R\ar[r]_{
  \begin{pmatrix}
    0\\\zeta_{\ell^{m}}
  \end{pmatrix}
}&\mathcal{O}^{\times}_{R}\oplus \mathcal{O}^{\times}_{R}\ar[d]^{
  \begin{pmatrix}
    -\ell&\ell^{m}
  \end{pmatrix}
}\\
&\mathcal{O}^{\times}_{R}}
\end{align*}
the last term being in degree $-1$.

Now, the domain and codomain of this morphism identifies with $R/\ell$ and $R/\ell(1)$ respectively:
\begin{align*}
  \xymatrix@C=80pt{&R\ar[d]^{
      \begin{pmatrix}
        -\ell^{m-1}\\
        -1
      \end{pmatrix}
}\\
R\ar[d]_{\ell}\ar[r]^{
  \begin{pmatrix}
    -1\\0
  \end{pmatrix}
}&R\oplus R\ar[d]_{
  \begin{pmatrix}
    -\ell&\ell^{m}
  \end{pmatrix}
}\ar[r]^{
        \begin{pmatrix}
          \zeta_{\ell^{m-1}}&0
        \end{pmatrix}
}&\mathcal{O}^{\times}_{R}\ar[d]_{
      \begin{pmatrix}
        -\ell^{m-1}\\
        -1
      \end{pmatrix}
}\\
R\ar[r]_{1}&R\ar[r]_{
  \begin{pmatrix}
    0\\\zeta_{\ell^{m}}
  \end{pmatrix}
}&\mathcal{O}^{\times}_{R}\oplus \mathcal{O}^{\times}_{R}\ar[d]^{
  \begin{pmatrix}
    -\ell&\ell^{m}
  \end{pmatrix}
}\ar[r]^{
           \begin{pmatrix}
             -1&\ell^{m-1}
           \end{pmatrix}
}&{\cal O}^{\times}_{R}\ar[d]^{\ell}\\
&&\mathcal{O}^{\times}_{R}\ar[r]_{1}&{\cal O}^{\times}_{R}}
\end{align*}
The composition is a model for $\beta_{\ell}$ using that $\zeta_{\ell^{m}}^{m-1}=\zeta_{\ell}$.
\end{proof}

\begin{cor}\label{bott-equality}
  Assume that $\F$ contains a primitive $\ell^{m+1}$th root of unity. Then in $\dm(\F,R)$ we have:
  \begin{equation*}
    \langle\Gamma_{*}\cone(\beta_{\ell^{m}})\rangle=\langle\gamma_{*}\cone(\beta_{\ell})\rangle
  \end{equation*}
\end{cor}
\begin{proof}
  The forward inclusion is \cref{bott-nN-integrally}. For the reverse inclusion we may invoke \cref{Nn} and obtain that
  \begin{equation*}
    \Gamma_{*}\Gamma^{*}\gamma_{*}\cone(\beta_{\ell})\in\langle\Gamma_{*}\cone(\beta_{\ell^{m}})\rangle.
  \end{equation*}
But $\Gamma_{*}\Gamma^{*}\gamma_{*}\cone(\beta_{\ell})=\Gamma_{*}\Gamma^{*}\Gamma_{*}\pi_{*}\cone(\beta_{\ell})=\gamma_{*}\cone(\beta_{\ell})\oplus\gamma_{*}\cone(\beta_{\ell})[1]$ by \cref{fbf}. 
We conclude that also $\gamma_{*}\cone(\beta_{\ell})\in\langle\Gamma_{*}\cone(\beta_{\ell^{m}})\rangle$.
\end{proof}

\section{Inverting the motivic Bott element}
\label{sec:bott}

In~\cite{haesemeyer-hornbostel:bott}, Haesemeyer and Hornbostel prove that under certain assumptions on $\F$, the étale sheafification functor
\begin{equation*}
  \aet:\dm(\F,\Z/n)\to\dmet(\F,\Z/n)
\end{equation*}
can be seen as the functor inverting the Bott element $\beta_{n}:\Z/n\to\Z/n(1)$. In other words, it induces an identification
\begin{equation*}
\dm(\F,\Z/n)[\beta_{n}^{-1}]\simeq  \dmet(\F,\Z/n).
\end{equation*}
Our goal in this section is to prove an analogous result with integral coefficients.

We first restate Haesemeyer-Hornbostel's result in a slightly improved form.
\begin{thm}[{\cite{haesemeyer-hornbostel:bott}}]\label{hh-bott} Let $\F$ be a field, and $n$ an integer. We assume:
  \begin{itemize}
  \item $n$ is prime to the characteristic of $\F$,
  \item $\F$ contains the $n$th roots
    of unity (respectively, and the 4th roots of unity if $n$ is even),
  \item $\F$ has finite étale $n$-dimension.
  \end{itemize}
Then étale sheafification induces equivalences of tensor triangulated categories
\begin{align*}
  \DMeff(\F,\Z/n)/\langle\cone(\beta_{n})\rangle^{\oplus}&\xrightarrow{\sim}\DMeteff(\F,\Z/n)\\
  \left(
  \dmeff(\F,\Z/n)/\langle\cone(\beta_{n})\rangle
  \right)^{\natural}&\xrightarrow{\sim}\dmeteff(\F,\Z/n)\\
  \DM(\F,\Z/n)/\langle\cone(\beta_{n})\rangle^{\oplus}&\xrightarrow{\sim}\DMet(\F,\Z/n)\\
  \left(
  \dm(\F,\Z/n)/\langle\cone(\beta_{n})\rangle
  \right)^{\natural}&\xrightarrow{\sim}\dmet(\F,\Z/n)
\end{align*}  
\end{thm}
\begin{proof}
  Haesemeyer and Hornbostel prove the first two equivalences under the additional assumption that $\F$ is perfect and admits resolution of singularities. These assumptions are used to apply Voevodsky's fundamental results on $\dm$. As we are inverting the exponential characteristic of $\F$, Kelly's \cite{kelly:dm-ldh} allows to remove the resolution of singularities assumption. Now let $\F$ be arbitrary and let $\F^{s}$ denote its inseparable closure. Then we have the following commutative square induced by scalar extension and étale sheafification:
  \begin{equation*}
    \xymatrix{\DMeff(\F,\Z/n)/\langle\cone(\beta_{n})\rangle^{\oplus}\ar[r]^{\sim}\ar[d]&\DMeff(\F^{s},\Z/n)/\langle\cone(\beta_{n})\rangle^{\oplus}\ar[d]^{\sim}\\
\DMeteff(\F,\Z/n)\ar[r]_{\sim}&\DMeteff(\F^{s},\Z/n)}
  \end{equation*}
  The right vertical arrow is an equivalence since $\F^{s}$ is perfect, the top horizontal arrow by \cite[8.1]{cisinski-deglise:integral-mixed-motives}, and the bottom horizontal arrow by \cite[6.3.16]{cisinski-deglise:etale-motives}. It follows that the left vertical arrow is an equivalence as well. This proves the first equivalence of the Theorem, and the second follows by restricting to the compact objects.

For the third equivalence we notice that both sides are compactly generated and the functor maps onto a set of compact generators. We therefore reduce to prove the fourth equivalence, or indeed that
\begin{equation*}
  \dm(\F,\Z/n)/\langle\cone(\beta_{n})\rangle
  \to\dmet(\F,\Z/n)
\end{equation*}
is fully faithful. For this consider the following commutative square
\begin{equation*}
  \xymatrix{\dmeff(\F,\Z/n)/\langle\cone(\beta_{n})\rangle\ar[r]\ar[d]&\dm(\F,\Z/n)/\langle\cone(\beta_{n})\rangle\ar[d]\\
\dmeteff(\F,\Z/n)\ar[r]&\dmet(\F,\Z/n)}
\end{equation*}
The left vertical arrow is fully faithful and the bottom horizontal arrow is an equivalence since $\Z/n\cong\Z/n(1)$. It therefore suffices to prove that the top horizontal arrow is an equivalence as well. This follows immediately from $\dm(\F,\Z/n)=\dmeff(\F,\Z/n)[(\otimes\Z/n(1))^{-1}]$ and again the fact that $\Z/n\cong\Z/n(1)$ after inverting $\beta_{n}$.
\end{proof}

\begin{rmk}
 The assumption on the cohomological dimension of $\F$ seems reasonable since for example the first equivalence in \cref{hh-bott} implies that the triangulated category $\DMeteff(\F,\Z/n)$ is compactly generated by smooth varieties.
\end{rmk}

\begin{con}
From now on we fix a field $\F$ and a set of primes $S$ containing the exponential characteristic of $\F$. We denote by $R$ the localization $S^{-1}\Z\subset\Q$. $\F$ is assumed to have finite $\ell$-cohomological dimension for every prime $\ell\notin S$, and to contain a primitive $\ell^{n}$th root of unity for every $\ell\notin S$ and every $n\geq 1$.
\end{con}

For any $\ell\notin S$, fix a primitive $\ell$th root of unity and interpret it as a morphism $\beta_{\ell}:\Z/\ell\to\Z/\ell(1)$ in $\dm(\F,\Z/\ell)$. We denote (abusively) the image of its cone in $\dm(\F,R)$ by $\ff\cone(\beta_{\ell})$.
\begin{thm}\label{bott-invert}

  The étale sheafification functor induces an equivalence of tt-categories
  \begin{equation*}
    \DM(\F,R)/\langle\ff\cone(\beta_{\ell})\mid\ell\notin S\rangle^{\oplus}\xrightarrow{\simeq}\DMet(\F,R).
  \end{equation*}
  In particular, it induces an equivalence on the level of geometric motives
  \begin{equation*}
    \left(
      \dm(\F,R)/\langle\ff\cone(\beta_{\ell})\mid\ell\notin S\rangle
\right)^{\natural}\xrightarrow{\simeq}\dmet(\F,R).
  \end{equation*}

  The same results hold for the effective versions.
\end{thm}
\begin{proof} In this proof, we denote by $\mathcal{D}$ the category $\DM(\F,R)$, by $\mathcal{D}^{\et}$ its étale version, by ${\cal B}$ (respectively ${\cal B}^{c}$) the localizing (respectively thick) subcategory of $\mathcal{D}$ generated by $\ff\cone(\beta_{\ell})$, $\ell\notin S$.

$\mathcal{D}$ is compactly generated, and the objects $\ff\cone(\beta_{\ell})$ are compact. It follows that the full subcategory of compact objects in the localization $\mathcal{D}/{\cal B}$ is canonically identified with the idempotent completion of the localization $\mathcal{D}^{c}/{\cal B}^{c}$. Thus the second statement follows from the first.
  
  The generators $R^{\et}(X)(n)$ ($X$ smooth, $n\in\Z$) for $\mathcal{D}^{\et}$ clearly lie in the image of $\aet$. It thus suffices to show that the functor in the first statement is fully faithful on compact objects. Let $M\in\mathcal{D}$ be a geometric motive, $N\in\mathcal{D}$ an arbitrary motive, and consider the triangle
  \begin{equation*}
    N\to N\otimes\Q\to N\otimes\Q/R\to N[1]
  \end{equation*}
  in $\mathcal{D}/{\cal B}$. By the long exact sequence associated to the functor $\hom_{\mathcal{D}/{\cal B}}(M,-)$ and the 5-lemma, it suffices to show that the following two maps are bijections (for arbitrary $N$):
  \begin{align}
    \label{bott-rational}\hom_{\mathcal{D}/{\cal B}}(M,N\otimes\Q)&\to\hom_{\mathcal{D}^{\et}}(M,N\otimes\Q),\\
    \label{bott-torsion}\hom_{\mathcal{D}/{\cal B}}(M,N\otimes\Q/R)&\to\hom_{\mathcal{D}^{\et}}(M,N\otimes\Q/R).
  \end{align}

  For the first map, consider the following diagram
  \begin{equation*}
    \xymatrix{\mathcal{D}\ar@<2pt>[r]^{\bc}\ar[d]_{\aet}&\mathcal{D}_{\Q}\ar[d]^{\aet}_{\sim}\ar@<2pt>[l]^{\ff}\\
      \mathcal{D}^{\et}\ar@<2pt>[r]^{\bc_{\et}}&\mathcal{D}^{\et}_{\Q}\ar@<2pt>[l]^{\ff^{\et}}}
  \end{equation*}
  where $\mathcal{D}^{(\et)}_{\Q}$ denotes the corresponding category of motives with rational coefficients. The subdiagram with the right adjoints removed is commutative. The left vertical map factors through the localization with respect to ${\cal B}$, and we obtain another diagram
  \begin{equation*}
    \xymatrix{\mathcal{D}\ar@<2pt>[r]^{\bc}\ar[d]_{q}&\mathcal{D}_{\Q}\ar@{=}[d]\ar@<2pt>[l]^{\ff}\\
\mathcal{D}/{\cal B}\ar@<2pt>[r]^{\bco}\ar[d]_{\overline{\aet}}&\mathcal{D}_{\Q}\ar[d]^{\aet}_{\sim}\ar@<2pt>[l]^{\ffo}\\
      \mathcal{D}^{\et}\ar@<2pt>[r]^{\bc_{\et}}&\mathcal{D}^{\et}_{\Q}\ar@<2pt>[l]^{\ff^{\et}}}
  \end{equation*}
  which commutes in the same sense. (Here, the existence of the right adjoint in the second row follows from Brown representability.)

  We claim that the following identities hold:
  \begin{align*}
    q\ff\bc=\ffo\bco q,&&\aet\ff\bc=\ff^{\et}\bc_{\et}\aet.
  \end{align*}
  To prove the first identity, let $A,B\in\mathcal{D}$ and consider the following sequence of canonical maps and identifications:
  \begin{align*}
    \hom_{\mathcal{D}/{\cal B}}(qA,\ffo\bco qB)&=\hom_{\mathcal{D}_{\Q}}(\bco qA,\bco qB)\\
                                                     &=\hom_{\mathcal{D}_{\Q}}(\bc A,\bc B)\\
                                                     &=\hom_{\mathcal{D}}(A,\ff\bc B)\\
                                                     &\to\hom_{\mathcal{D}/{\cal B}}(qA,q\ff\bc B)
  \end{align*}
  We need to show that the last map (induced by the localization) is bijective, which means that $\ff\bc B$ is local with respect to the localization, in other words $\ff\bc B\in({\cal B})^{\bot}$. But for any prime $\ell\notin S$ and any integer $p$, we have
  \begin{align*}
    \hom_{\mathcal{D}}(\ff\cone(\beta_{\ell})[p],\ff\bc B)=\hom_{\mathcal{D}_{\Q}}(\bc\ff\cone(\beta_{\ell})[p],\bc B)=\hom_{\mathcal{D}_{\Q}}(0,\bc B)=0
  \end{align*}
  and thus the claim.

  For the second identity, recall that in $\mathcal{D}$ (respectively $\mathcal{D}^{\et}$), the composition $\ff\bc$ (respectively $\ff^{\et}\bc_{\et}$) is simply tensoring with $\Q$. Since $\aet$ is monoidal, we get indeed
  \begin{align*}
    \aet\ff\bc A=\aet(A\otimes\Q)=\aet A\otimes\aet\Q=\aet A\otimes\Q=\ff^{\et}\bc_{\et}\aet A.
  \end{align*}

  Let us come back to the map in \cref{bott-rational}. It decomposes as the following composition:
  \begin{align}\label{etale-bijection}
    \begin{split}
      \hom_{\mathcal{D}/{\cal B}}(qM,q\ff\bc N)&=\hom_{\mathcal{D}/{\cal B}}(qM,\ff\bc qN)\\
      &=\hom_{\mathcal{D}_{\Q}}(\bc M,\bc N)\\
      &=\hom_{\mathcal{D}^{\et}_{\Q}}(\aet\bc M,\aet\bc N)\\
      &=\hom_{\mathcal{D}^{\et}_{\Q}}(\bc_{\et}\aet M,\bc_{\et}\aet N)\\
      &=\hom_{\mathcal{D}^{\et}}(\aet M,\ff^{\et}\bc_{\et}\aet N)\\
      &=\hom_{\mathcal{D}^{\et}}(\aet M,\aet\ff\bc N)
    \end{split}
  \end{align}
  and is therefore a bijection.

  We now turn to \cref{bott-torsion}. The motive $\Q/R$ is a sum $\oplus_{\ell\notin S}R[\ell^{\infty}]$ where $R[\ell^{\infty}]$ is the homotopy colimit of $(R/\ell^{n})_{n}$ with transition maps $R/\ell^{n}\to R/\ell^{n+1}$ given by multiplication by $\ell$. As the tensor product commutes with direct sums and homotopy colimits and $M$ is compact, we reduce to show that
  \begin{equation}
    \label{bott-finite}
    \hom_{\mathcal{D}/{\cal B}}(M,N\otimes R/\ell^{n})\to\hom_{\mathcal{D}^{\et}}(M,N\otimes R/\ell^{n})
  \end{equation}
  is a bijection.

  Write $\mathcal{D}^{(\et)}_{\ell^{n}}$ for the corresponding categories of motives with $R/\ell^{n}$-coefficients. Let ${\cal B}_{\ell^{n}}$ be the localizing subcategory of $\mathcal{D}_{\ell^{n}}$ generated by $\cone(\beta_{\ell^{n}})$. By \cref{Nn}, $\bc$ maps ${\cal B}$ to ${\cal B}_{\ell^{n}}$ so that we obtain a diagram
  \begin{equation*}
    \xymatrix{\mathcal{D}\ar@<2pt>[r]^{\bc}\ar[d]_{q}&\mathcal{D}_{\ell^{n}}\ar[d]^{q}\ar@<2pt>[l]^{\ff}\\
\mathcal{D}/{\cal B}\ar@<2pt>[r]^{\bco}\ar[d]_{\overline{\aet}}&\mathcal{D}_{\ell^{n}}/{\cal B}_{\ell^{n}}\ar[d]^{\overline{\aet}}_{\sim}\ar@<2pt>[l]^{\ffo}\\
      \mathcal{D}^{\et}\ar@<2pt>[r]^{\bc_{\et}}&\mathcal{D}^{\et}_{\ell^{n}}\ar@<2pt>[l]^{\ff^{\et}}}
  \end{equation*}
as before. The bottom right equivalence follows from \cref{hh-bott}.

We now claim that the analogous identities hold:
\begin{align*}
  q\ff\bc=\ffo\bco q,&&\aet\ff\bc=\ff^{\et}\bc_{\et}\aet.
\end{align*}
The second identity is proved as before, while the first one can be verified as follows. Since $\bc$ preserves compact objects $\ff$ commutes with small sums, and it follows that $\ff$ maps ${\cal B}_{\ell^{n}}$ to ${\cal B}$, by \cref{bott-nN-integrally}. It follows from \cref{triangulated-adjunction} that $\bco$ and $\ffo$ are simply the functors induced by $\bc$ and $\ff$, respectively. The first identity follows immediately. The map in \cref{bott-finite} is now seen to be a bijection precisely as in \cref{etale-bijection}.
\end{proof}

\begin{lem}\label{triangulated-adjunction}
  Let ${\cal D},{\cal D}'$ be two triangulated categories, let $\bc\dashv\ff:{\cal D}\to{\cal D}'$ be an adjunction, and suppose ${\cal B}\subset{\cal D}$ and ${\cal B}'\subset{\cal D}'$ are thick subcategories such that the corresponding Bousfield localizations exist.

If $\bc{\cal B}\subset{\cal B}'$ and $\ff{\cal B}'\subset{\cal B}$ then $\bc$ and $\ff$ descend to an adjunction on the quotient categories: $\bc\dashv\ff:{\cal D}/{\cal B}\to {\cal D}'/{\cal B}'$.
\end{lem}
\begin{proof}
  By our assumption, these functors do descend to the quotient categories. To see that the induced functors are still adjoint to each other, let us denote by $L'$ the localization functor on ${\cal D}'$ with respect to ${\cal B}'$, and consider the following sequence of morphisms and canonical identifications, where $a\in{\cal D}$ and $b\in{\cal D}'$:
  \begin{align*}
    \hom_{{\cal D}'/{\cal B}'}(\bc a,b)&=\hom_{{\cal D}'/{\cal B}'}(\bc a,L'b)\\
    &=\hom_{{\cal D}'}(\bc a,L'b)\\
    &=\hom_{{\cal D}}( a,\ff L'b)\\
    &\to \hom_{{\cal D}/{\cal B}}(a,\ff L'b)\\
    &\leftarrow \hom_{{\cal D}/{\cal B}}(a,\ff b)
  \end{align*}
It remains to check that the last two maps are bijections. We know that $L'b\in({\cal B}')^{\bot}$ and hence for any $x\in{\cal B}$ we have
\begin{align*}
  \hom_{{\cal D}}(x,\ff L'b)=\hom_{{\cal D}'}(\bc x,L'b)=0
\end{align*}
since $\bc{\cal B}\subset{\cal B}'$. In other words $\ff L'b\in{\cal B}^{\bot}$ and the first map above is therefore a bijection. We also know that $b\to L'b$ becomes invertible in ${\cal D}'/{\cal B}'$, \ie its cone lies in ${\cal B}'$. It follows that the cone of $\ff b\to\ff L'b$ lies in $\ff{\cal B}'\subset{\cal B}$ and the second map above is therefore a bijection as well.

\end{proof}

\end{document}